\documentclass[11pt,a4paper,leqno]{article}
\usepackage{mathrsfs,exscale,ifthen,amsfonts,amsmath,amscd,amssymb,amsthm,dsfont}
\usepackage{enumerate}
\usepackage{bbold}
\usepackage{graphicx}
\usepackage[T1]{fontenc}
\usepackage[latin1]{inputenc}
\usepackage[left=2cm,right=2cm,top=2cm,bottom=2cm]{geometry}
\usepackage{color}
\usepackage{hyperref}
\usepackage{authblk}
\usepackage{accents}

\DeclareMathOperator{\supp}{supp}
\DeclareMathOperator{\osc}{\mbox{osc}}

\DeclareMathOperator{\defeq}{\mathrel{\mathop:}=}

\newtheorem{theorem}{Theorem}[section]
\newtheorem{lemma}[theorem]{Lemma}
\newtheorem{proposition}[theorem]{Proposition}
\newtheorem{definition}{Definition}[section]
\newtheorem{corollary}[theorem]{Corollary}
\newtheorem{remark}[theorem]{Remark}

\numberwithin{equation}{section}
\numberwithin{theorem}{section}

\newcommand{\F}{\mathcal{F}}
\newcommand{\Ew}{\mathcal{E}^\omega}
\newcommand{\E}{\mathcal{E}}
\newcommand{\mean}{\mathbb{E}}
\newcommand{\PR}{\mathbb{P}}
\newcommand{\N}{\mathbb{N}}
\newcommand{\R}{\mathbb{R}}

\newcommand{\be}{\begin{equation}}
\newcommand{\ee}{\end{equation}}

\makeatletter
\def\namedlabel#1#2{\begingroup
   \def\@currentlabel{#2}%
   \label{#1}\endgroup
}
\makeatother

\author{Chiarini, Alberto\thanks{
		Institut f{\"u}r Mathematik, Technische Universit{\"a}t Berlin, Stra{\ss}e des 17. Juni 136, 10623 Berlin, Germany. Email: chiarini@math.tu-berlin.de}
	\ and Deuschel, Jean-Dominique\thanks{ Institut f{\"u}r Mathematik, Technische Universit{\"a}t Berlin, Stra{\ss}e des 17. Juni 136, 10623 Berlin, Germany. Email: deuschel@math.tu-berlin.de}}
\title{Local Central Limit Theorem for diffusions in a degenerate and unbounded Random Medium.}

\begin{document}

\maketitle
\begin{abstract} We study a symmetric diffusion $X$ on $\R^d$ in divergence form in a stationary and ergodic environment, with measurable unbounded and degenerate coefficients. We prove  a quenched local central limit theorem for $X$, under some moment conditions on the environment; the key tool is a local parabolic Harnack inequality obtained with Moser iteration technique.\newline
\end{abstract}

\section{Description of the Main Result}

We model the stationary and ergodic random environment by a probability space $(\Omega,\mathcal{G},\mu)$, on which we define a measure-preserving group of transformations $\tau_x:\Omega\to \Omega$, $x\in\R^d$. One can think about $\tau_x\omega$ as a translation of the environment $\omega\in\Omega$ in direction $x\in \R^d$. The function $(x,\omega)\to \tau_x\omega$ is assumed to be $\mathcal{B}(\R^d)\otimes\mathcal{G}$-measurable and such that 
if $\tau_x A = A$ for all $x\in\R^d$, then $\mu(A)\in\{0,1\}$.
Given the random environment $(\Omega,\mathcal{G},\mu,\{\tau_x\}_{x\in\R^d})$ we can construct a stationary and ergodic random field simply taking a random variable $f:\Omega\to \R$ and defining $f^\omega(x):=f(\tau_x\omega)$, $x\in \R^d$.

We are given a $\mathcal{G}$-measurable function $a:\Omega\to \R^{d\times d}$ such that 
\begin{itemize}
\item[$(a.1)$\namedlabel{ass:a.1}{$(a.1)$}]there exist $\mathcal{G}$-measurable non-negative functions $\lambda,\Lambda:\Omega\to\R$ such that for $\mu$-almost all $\omega\in\Omega$ and all $\xi\in\R^d$
 \[
  \lambda(\omega)|\xi|^2\leq \langle a(\omega) \xi,\xi\rangle \leq \Lambda(\omega)|\xi|^2,
 \]
\item[$(a.2)$\namedlabel{ass:a.2}{$(a.2)$}]there exist $p,q\in[1,\infty]$ satisfying  $1/p+1/q<2/d$ such that
 \[
\mean_\mu[\Lambda^p]<\infty,\quad \mean_\mu[\lambda^{-q}]<\infty.
\]
\end{itemize}

Our diffusion process is formally associated with the following generator in divergence form
\begin{equation}\label{eq:formal}
L^{\omega} u(x) = \frac{1}{\Lambda^\omega(x)}\nabla\cdot(a^\omega(x)\nabla u(x)).
\end{equation}
Since $a^{\omega}(x)$ is modeling a random field, it is not natural to assume its differentiability in $x\in\R^d$. Therefore the operator defined in \eqref{eq:formal} does not make sense, and the standard techniques from the Stochastic Differential Equations theory or It\^o calculus are not helpful nor in the construction of the diffusion process nor in performing the relevant computations.

We will exploit Dirichlet Forms theory to construct the diffusion process formally associated with \eqref{eq:formal}. Instead of the operator $L^\omega$ we shall consider the bilinear form obtained by $L^\omega$ formally integrating by parts,
\begin{equation}\label{eq:df}
\Ew(u,v) = \sum_{i,j}\int_{\R^d} a_{ij}^\omega(x)\partial_i u(x)\partial_j u(x) dx
\end{equation}
for a proper class of functions $u,v\in\F^{\Lambda,\omega}\subset L^2(\R^d,\Lambda^\omega dx)$, more precisely $\F^{\Lambda,\omega}$ is the closure of $C_0^\infty(\R^d)$ in $L^2(\R^d,\Lambda^\omega dx)$ with respect to $\Ew+(\cdot,\cdot)_{\Lambda}$. It is a classical result of Fukushima \cite{fukushima1994dirichlet} that it is possible to associate to \eqref{eq:df} a diffusion process $(X^\omega, \PR_x^\omega)$ as soon as $(\lambda^\omega)^{-1}$ and $\Lambda^\omega$ are locally integrable. As a drawback, the process  cannot in general start from every $x\in\R^d$ but only from almost all, and the set of exceptional points may depend on the realization of the environment.

In \cite{chiarinideuschel} it was proved that if $\lambda^\omega(\cdot)^{-1},\Lambda^\omega(\cdot)\in L^\infty_{loc}(\R^d)$ for $\mu$-almost all $\omega\in\Omega$
then a quenched invariance principle holds for $X^\omega$, namely the scaled process $X^{\epsilon,\omega}_t\defeq \epsilon X^\omega_{t/\epsilon^2}$ converges in distribution under $\PR_0^\omega$ to a Brownian motion with a non-trivial deterministic covariance structure as $\epsilon\to 0$. In that work local boundness was assumed in order to get some regularity for the density of the process $X^\omega$ and avoid technicalities due to exceptional sets arising from Dirichlet forms theory. 

In this paper we show that if a quenched invariance principle holds, then under \ref{ass:a.1} and \ref{ass:a.2}, the density of $X^{\epsilon,\omega}$ converges uniformly on compacts to the gaussian density. Hence, to state the theorem we need the following assumption.
\begin{itemize}
\item[$(a.3)$\namedlabel{ass:a.3}{$(a.3)$}] Assume that there is a positive definite symmetric $d$-dimensional matrix $\Sigma$ such that for $\mu$-almost all $\omega\in\Omega$ we have that for almost all $o\in\R^d$, all balls $B\subset \R^d$ and all compact intervals $I\subset (0,\infty)$
\[
\lim_{\epsilon\to 0} \PR_o^\omega(\epsilon X^\omega_{t/\epsilon^2}\in B) = \frac{1}{\sqrt{(2\pi t)^{d} \det \Sigma}}\int_B  \exp\Bigl(-\frac{x\cdot \Sigma^{-1} x}{2t}\Bigr)\,dx
\]
uniformly in $t\in I$.
\end{itemize}

Observe that provided that $\lambda^\omega(\cdot)^{-1},\Lambda^\omega(\cdot)\in L^\infty_{loc}(\R^d)$ for $\mu$-almost all $\omega\in\Omega$ then assumption \ref{ass:a.3} is satisfied for all $o\in\R^d$, $\mu$-almost surely due to Theorem 1.1 in \cite{chiarinideuschel}.

\begin{theorem}\label{thm:localCLTRandom}
Let $d\geq 2$. Assume \ref{ass:a.1}, \ref{ass:a.2} and \ref{ass:a.3}. Let $p_t^\omega(\cdot,\cdot)$ be the density with respect to $\Lambda^\omega(x)dx$ of the semigroup $P_t^\omega$ associated to $(\Ew, \F^{\Lambda,\omega})$ on $L^2(\R^d,\Lambda^\omega dx)$. Let $r>0$ and $I\subset (0,\infty)$ compact. Then for  $\mu$-almost all $\omega\in\Omega$ we have that for almost all $o\in\R^d$
\begin{equation}\label{eq:localclt}
\lim_{\epsilon\to 0} \sup_{|x-o|\leq r} \sup_{t\in I}|\epsilon^{-d} p^\omega_{t/\epsilon^2}(o,x/\epsilon)-\mean_\mu[\Lambda]^{-1}k_t^\Sigma(x)|=0.
\end{equation}
If we assume further that $\lambda^\omega(\cdot)^{-1},\Lambda^\omega(\cdot)\in L^\infty_{loc}(\R^d)$ for $\mu$-almost all $\omega\in\Omega$ then \eqref{eq:localclt} is satisfied for all $o\in\R^d$.
\end{theorem}

\paragraph{The method.} The proof of Theorem \ref{thm:localCLTRandom} relies strongly on a priori estimates for solutions to the ``formal'' parabolic equation
\begin{equation}\label{eq:formalpar}
\partial_t u(t,x)-\frac{1}{\Lambda^\omega(x)}\nabla\cdot (a^\omega(x)\nabla u(t,x)) = 0,\qquad t\in(0,\infty),\,x\in\R^d.
\end{equation}
It is well known that when $x\to a^\omega(x)$ and $x\to \Lambda^\omega(x)$ are bounded and bounded away from zero, uniformly in $\omega\in\Omega$, then a parabolic Harnack's inequality holds for solutions to \eqref{eq:formalpar}, this is a celebrated result due to Moser \cite{moser1964parabolic}. He showed that there is a positive constant $C_{PH}$, which depends only on the uniform bounds on $a$ and $\Lambda$, such that for any positive weak solution of \eqref{eq:formalpar} on $(t,t+r^2)\times B(x,r)$ we have
\[
\sup_{(s,z)\in Q_-} u(s,z) \leq C_{PH} \inf_{(s,z)\in Q_+} u(s,z)
\]
where $Q_-=(t+1/4 r^2,t+1/2 r^2)\times B(x,r/2)$ and $Q_+=(t+3/4 r^2,t+ r^2)\times B(x,r/2)$. The parabolic Harnack inequality plays a prominent role in the theory of partial differential equations, in particular to prove H\"older continuity for solutions to parabolic equations, as it was observed by Nash \cite{nash1958} and De Giorgi \cite{zbMATH03138423}, or to prove Gaussian type bounds for the fundamental solution $p_t^\omega(x,y)$ of \eqref{eq:formalpar} as done by Aronson \cite{aronson1967}. It is remarkable that such results do not depend neither on the regularity of $a$ nor of $\Lambda$.

In this paper we shall exploit the stability of Moser's method to derive a parabolic Harnack inequality also in the case of degenerate and possibly unbounded coefficients. The technique is quite flexible and can also be applied to discrete space models for which we refer to \cite{deuschelslowikandresharnack}
. 

Moser's method is based on two steps. One wants first to get a Sobolev inequality to control some $L^\rho$ norm in terms of the Dirichlet form and then control the Dirichlet form of any caloric function by a lower moment. This sets up an iteration which leads to bound the $L^\infty$ norm of the caloric function. In the uniform elliptic case this is rather standard and it is possible to control the $L^{2d/d-2}$ norm by the $L^2$ norm. In our case the  coefficients are neither bounded from above nor from below and we need to work with a weighted Sobolev inequality, which was already established in \cite{chiarinideuschel} by means of H\"older's inequality. Doing so we are able to control locally on balls the $L^\rho$ norm by means of the $L^{2p^*}$ norm, with $\rho = 2qd/[q(d-2)+d]$. In order to start the iteration we need $\rho>2 p^*$ which is equivalent to $1/p+1/q<2/d$. This integrability assumption firstly appeared in \cite{edmundspeletier} in order to extend the results of De Giorgi and Nash to degenerate elliptic equations, although they focus on weights belonging to the Muckenhaupt's class.  A similar condition was also recently exploited in \cite{Zhikov} to obtain Aronson type estimates for solutions to degenerate parabolic equations.

Following the classic proof of Moser, with some extra care due to the different exponents we get a parabolic Harnack inequality for solution to \eqref{eq:formalpar} in our setting. In the uniform elliptic and bounded case the constant in front of the Harnack inequality was depending only on uniform bounds on $a$ and $\Lambda$. In our setting we cannot expect that to be true for general weights, and the constant will strongly depend on the center and the radius of the ball, in particular we don't have any control for small balls, so that a genuine H\"older's continuity result like the one of Nash is not given.  Luckily in the diffusive limit the ergodic theorem helps to control constants and to give Theorem \ref{thm:localCLTRandom}.

\begin{remark}Given a speed measure $\theta:\Omega \to (0,+\infty)$ one can consider also the Dirichlet form $(\Ew,\F^{\theta,\omega})$ on $L^2(\R^d,\theta^\omega dx)$ where $\Ew$ is given by \eqref{eq:df} and $\F^{\theta,\omega}$ is the closure of of $C_0^\infty(\R^d)$ in $L^2(\R^d,\theta^\omega dx)$ with respect to $\Ew+(\cdot,\cdot)_{\theta}$. This corresponds to the formal generator
\[
L^{\omega} u(x) = \frac{1}{\theta^\omega(x)}\nabla\cdot(a^\omega(x)\nabla u(x)).
\]
One can show along the same lines of the proof for $\theta = \Lambda$ that if
\[
\mean_\mu[\theta^r]<\infty,\, \mean_\mu[\lambda^{-q}]<\infty,\, \mean_\mu[\Lambda^p \theta^{1-p}]<\infty,
\]
where $p,q,r\in(1,\infty]$ are such that
\[
\frac{1}{r}+\frac{1}{q}+\frac{1}{p-1}\frac{r-1}{r}<\frac{2}{d},
\]
then the parabolic Harnack inequality still works, in particular a quenched local central limit theorem can still be derived in this situation.

Observe that in the case $\theta=\Lambda$ we find back the familiar condition $1/p+1/q <2/d$. In the case that $\theta \equiv 1$, $r=\infty$ the condition reads $1/(p-1)+1/q<2/d$.
\end{remark}

\begin{remark} The condition $1/p+1/q<2/d$ is morally optimal to state Theorem \ref{thm:localCLTRandom}. Indeed it was shown in \cite{deuschelslowikandresharnack}[See Theorem 5.4] that if $1/p+1/q>2/d$, then there is an ergodic environment for which the quenched local central limit theorem does not hold. It is not hard construct an example also in the continuous by exploiting the same ideas given in \cite{deuschelslowikandresharnack}.

\end{remark}

A summary of the paper is the following. In Section 2 we present a deterministic model obtained by looking at a fixed realization of the environment. We derive Sobolev, Poincar\'e and Nash inequalities for such a model.

In Section 3 we prove a priori estimates, on-diagonal bounds and H\"older continuity type estimates for caloric functions. The main aim and result of the section is the parabolic Harnack inequality. 

In section 4 we prove a local Central Limit Theorem for the deterministic model which we apply to finally get Theorem \ref{thm:localCLTRandom}.
\section{Deterministic Model and Local inequalities}

Since we want to prove a quenched result we will develop a collection of inequalities for a deterministic model. With a slight abuse of notation we will note with $a(x)$, $\lambda(x)$ and $\Lambda(x)$ the deterministic versions of $a(\tau_x\omega)$, $\lambda(\tau_x\omega)$ and $\Lambda(\tau_x\omega)$.

We are given a symmetric matrix $a:\R^d\to \R^{d\times d}$ such that 
\begin{itemize}
\item[$(b.1)$\namedlabel{ass:b.1}{$(b.1)$}]there exist $\lambda,\Lambda:\R^d\to\R$ non-negative such that for almost all $x\in\R^d$ and $\xi\in\R^d$
 \[
  \lambda(x)|\xi|^2\leq \langle a(x) \xi,\xi\rangle \leq \Lambda(x)|\xi|^2,
 \]
\item[$(b.2)$\namedlabel{ass:b.2}{$(b.2)$}]there exist $p,q\in[1,\infty]$ satisfying  $1/p+1/q<2/d$ such that
 \[
  \limsup_{r\to \infty} \frac{1}{|B(0,r)|} \int_{B(0,r)} \Lambda^p +\lambda^{-q} \,dx <\infty.
 \]
\end{itemize}
Assumption \ref{ass:b.2} plays the role of ergodicity in the random environment model.

We are interested in finding a priori estimates for solutions to the formal parabolic equation
\begin{equation}\label{eq:par}
\partial_t u(t,x)-\frac{1}{\Lambda(x)}\nabla\cdot (a(x)\nabla u(t,x)) = 0,
\end{equation}
for $t\in (0,\infty)$ and $x\in \R^d$.

Clearly in the way it is stated \eqref{eq:par} is not well defined since $a$ is only assumed to be measurable. In order to make sense of \eqref{eq:par} we shall exploit the Dirichlet form framework, see \cite{fukushima1994dirichlet} for an exhaustive treatment on the subject. 

\subsection{Caloric Functions}
For this section we will follow \cite{bgk2f}. Let $\theta:\R^d\to\R$ be a non-negative function such that $\theta^{-1},\theta$ are locally integrable on $\R^d$.
Consider the symmetric form $\E$ on $L^2(\R^d,\theta dx)$ with domain $C_0^\infty(\R^d)$ defined by
\begin{equation}\label{eq:DF}
\E(u,v)\defeq\sum_{i,j}\int_{\R^d} a_{ij}(x)\partial_i u(x)\partial_j v(x)\,dx.
\end{equation}
Then, $(\E,C_0^\infty(\R^d))$ is  \emph{closable} in $L^2(\R^d,\theta dx)$ thanks to \cite{rockner}[Ch. II example 3b], since $\lambda^{-1},\Lambda\in L^1_{loc}(\R^d)$ by \ref{ass:b.2}. We shall denote by $(\E,\F^\theta)$ such a closure; it is clear that $\F^\theta$ is the completion of $C_0^\infty(\R^d)$ in $L^2(\R^d,\theta dx)$ with respect to $\E_1:=\E+(\cdot,\cdot)_\theta$. Observe that $(\E,\F^\theta)$ is a strongly local regular Dirichlet form, having $C_0^\infty(\R^d)$ as a core. In the case that $\theta\equiv 1$ we will simply write $\F$. Given an open subset $G$ of $\R^d$ we will denote by $\F_G^\theta$ the closure of $C_0^\infty(G)$ in $L^2(G,\theta dx)$ with respect to $\E_1$.

\begin{definition}[Caloric functions] Let $I\subset \R$ and $G\subset \R^d$ an open set. We say that a function $u : I \to \F^\theta$ is a subcaloric (supercaloric) function in $I\times G$ if $t\to(u(t,\cdot),\phi)_{\theta}$ is differentiable in $t\in I$ for any $\phi\in L^2(G,\theta dx)$ and
\begin{equation}\label{eq:parabolic}
\frac{d}{dt}(u,\phi)_{\theta}+\E(u,\phi)\leq 0,\quad (\geq)
\end{equation}
for all non negative $\phi\in \F_G^\Lambda$.  We say that a function $u : I \to \F^\theta$ is a caloric function in $I\times G$ if it is both sub- and supercaloric.
\end{definition}

It is clear from the definition that if a function is subcaloric on $I\times G$ than it is caloric on $I'\times G'$ whenever $I'\subset I$ and $G'\subset G$.

Moreover, observe that if $P_t^G$ is the semigroup associated to $(\E,\F^\theta)$ on $L^2(G, \theta dx)$ and $f\in L^2(G, \theta dx)$, for a given open set $G\subset \R^d$, then the function $u(t,\cdot)=P_t^G f(\cdot)$ is a caloric function on $(0,\infty)\times G$. To complete the picture we state the following maximum principle which appeared in \cite{grig2009}. For a real number $a$ denote by $a_+ = a \vee 0$. 

\begin{lemma} Fix $T\in (0,\infty]$, a set $G\subset \R^d$ and let $u :(0,T)\to \F^\theta_G$ be a subcaloric function in $(0,T)\times G$ which satisfies the boundary condition
$u_+(t,\cdot)\in \F^\theta_G$, $\forall t\in(0,T)$ 
and $u_+(t,\cdot)\to 0$ in $L^2(G,\theta dx)$ as $t\to 0$. Then $u\leq 0$ on $(0,T)\times G$.
\end{lemma}

As a corollary of this lemma we have the super-mean value inequality for subcaloric functions.
\begin{corollary}Fix $T\in (0,\infty]$, an open set $G\subset \R^d$ and $f\in L^2(G,\theta dx)$ non-negative. Let $u :(0,T)\to \F^\theta_G$ be a non-negative subcaloric function on $(0,T)\times G$ such that $u(t,\cdot)\to f$ in $L^2(G,\theta dx)$ as $t\to0$. Then for any $t\in(0,T)$
\[
u(t,\cdot)\geq P_t^G f,\ \text{in $G$}.
\]
In particular for $0< s< t<T$
\[
u(t,\cdot)\geq P_{t-s}^G u(s,\cdot),\ \text{in $G$}.
\]
\end{corollary}

\subsection{Sobolev inequalities}

In this section we will state local inequalities on the flat space $L^2(\R^d,dx)$ and on the weighted space $L^2(\R^d,\Lambda dx)$. We are interested in Sobolev, Poincar\'e and Nash type inequalities. The first and the second provide an effective tool for deriving local estimates on solutions to Elliptic and Parabolic degenerate partial differential equation, while the latter will be used to prove the existence of a kernel for the semigroup $P_t$ associated to $(\E,\F^\Lambda)$ on $L^2(R^d,\Lambda dx)$. 

We shall see that the constants appearing in the inequalities are strongly dependent on averages of $\lambda$ and $\Lambda$ and in particular on the ball where we focus our analysis. 

\paragraph{Notation.} Let $B\subset \R^d$ be a bounded set. For a function $u:B\to\R$, $r\geq 1$ and a weight $\theta:B\to \R$ we note
\[
\|u\|_{r,\theta} \defeq \biggl(\int_{\R^d} |u(x)|^r\theta(x)dx\biggr)^\frac{1}{r},\quad\|u\|_{r,B} \defeq \biggl(\frac{1}{|B|}\int_B |u(x)|^r\,dx\biggr)^\frac{1}{r}.
\]
and
\[
\|u\|_{r,B,\theta} \defeq \biggl(\frac{1}{|B|}\int_B |u(x)|^r\,\theta(x) dx\biggr)^\frac{1}{r}.
\]
In the sequel we shall use the symbol $\lesssim$ to say that the inequality $\leq$ holds up to a multiplicative constant depending only on the dimension $d\geq 2$.

In the next proposition it is enough to assume  $\Lambda\in L^1_{loc}(\R^d)$ and $\lambda^{-1}\in L^q_{loc}(\R^d)$.  The following constant will play an important role in the sequel,
\begin{equation}\label{eq:rho}
\rho(q,d):= \frac{2qd}{q(d-2)+d},
\end{equation}
observe that $\rho$ is the Sobolev's conjugate of $2q/(q+1)$.

\begin{proposition}[Local Sobolev inequality] Fix a ball $B\subset\R^d$. Then for all $u\in\F_B$
\begin{equation}\label{eq:localsobolev}
\|u\|_{\rho,B}^2\lesssim C_{S}^B |B|^\frac{2}{d}\,\frac{\E(u,u)}{|B|},
\end{equation}
where $C_S^B\defeq \|\lambda^{-1}\|_{q,B}$.
\end{proposition}
\begin{proof}  We start proving \eqref{eq:localsobolev} for $u\in C_0^\infty(B)$. Since $\rho$ as defined in \eqref{eq:rho} is the Sobolev conjugate of $2q/(q+1)$, by the classical Sobolev's inequality 
\[
\|u\|_\rho\lesssim \|\nabla u\|_{2q/(q+1)},
\]
where it is clear that we are integrating over $B$. By H\"older's inequality and \ref{ass:b.1} we can estimate the right hand side as follows
\[
\|\nabla u\|_{2q/(q+1)}^2=\Bigl(\int_B|\nabla u|^\frac{2q}{q+1}\lambda^\frac{q}{q+1}\lambda^{-\frac{q}{q+1}}\,dx\Bigr)^\frac{q+1}{q}\leq \|1_B\lambda^{-1}\|_{q}\,\E(u,u),
\]
which leads to \eqref{eq:localsobolev} for $u\in C_0^\infty(B)$ after averaging over the ball $B$. By approximation, the inequality is easily extended to $u\in\F_B$.
\end{proof}

\begin{proposition}[Local weighted Sobolev inequality] Fix a ball $B\subset\R^d$. Then for all $u\in\F_B^\Lambda$
\begin{equation}\label{eq:localweightedsobolev}
\|u\|_{\rho/p^*,B,\Lambda}^2\lesssim C^{B,\Lambda}_S|B|^\frac{2}{d} \, \frac{\E(u,u)}{|B|},
\end{equation}
being $C^{B,\Lambda}_S\defeq\|\lambda^{-1}\|_{q,B}\|\Lambda\|_{p,B}^{2p^*/\rho}$ and $p^*=p/(p-1)$.
\end{proposition}
\begin{proof} The proof easily follows from H\"older's inequality
\[
\|u\|_{\rho/p^*,B,\Lambda}^2\leq \|u\|_{\rho,B}^2 \|\Lambda\|_{p,B}^{2p^*/\rho}
\]
and the previous proposition. 
\end{proof}

\begin{remark}
From these two Sobolev's inequalities it follows that the domains $\F_B$ and $\F_B^\Lambda$ coincide for all balls $B\subset \R^d$. Indeed, from \eqref{eq:localsobolev} and \eqref{eq:localweightedsobolev}, since $\rho,\rho/p^*>2$, we get that $(\F_B,\E)$ and $(\F_B^\Lambda,\E)$ are two Hilbert spaces; therefore $\F_B,\F_B^\Lambda$ coincide with their extended Dirichlet space, which by \cite[pag 324]{Fukushima1987}, is the same, hence $\F_B = \F_B^\Lambda$. 
\end{remark}

\paragraph{Cutoffs.}Since assumptions \ref{ass:b.1} and \ref{ass:b.2} only assure local integrability of $\lambda^{-1}$ and $\Lambda$, we will need to work with functions that are locally in $\F$ or $\F^\Lambda$ and with cutoff functions.

Let $B\subset\R^d$ be a ball, a cutoff on $B$ is a function $\eta\in C_0^\infty(B)$, such that $0\leq \eta\leq 1$. 
Given $\theta:\R^d\to\R$ as before, we say that $u\in\F_{loc}^\theta$, if for all balls $B\subset\R^d$ there exists $u_B\in\F^\theta$ such that $u \equiv u_B$ almost surely on $B$.

In view of these notations, for $u,v\in \F^\theta_{loc}$ we define the bilinear form
\begin{equation}\label{eq:dfcutoff}
\E_\eta(u,v)=\sum_{i,j}\int_{\R^d} a_{ij}(x)\partial_i u(x)\partial_j v(x)\eta^2(x)\,dx
\end{equation}

\begin{proposition}[Local Sobolev inequality with cutoff] Fix a ball $B\subset\R^d$ and a cutoff function $\eta\in C_0^\infty (B)$ as above. Then for all $u\in\F_{loc}^\Lambda\cup\F_{loc}$
\begin{equation}\label{eq:localsobolev_cutoff}
\|\eta u\|_{\rho,B}^2\lesssim C_{S}^B|B|^\frac{2}{d} \Bigl[\frac{\E_\eta(u,u)}{|B|}+\|\nabla \eta\|_\infty^2\| u\|^2_{2,B,\Lambda}\Bigr],
\end{equation}
and, for the weighted version
\begin{equation}\label{eq:localweightedsobolev_cutoff}
\|\eta u\|_{\rho/p^*,B,\Lambda}^2\lesssim C_{S}^{B,\Lambda} |B|^\frac{2}{d} \Bigl[\frac{\E_\eta(u,u)}{|B|}+\|\nabla \eta\|_\infty^2\| u\|^2_{2,B,\Lambda}\Bigr].
\end{equation}
\end{proposition}
\begin{proof} 
We prove only $\eqref{eq:localsobolev_cutoff}$, being \eqref{eq:localweightedsobolev_cutoff} analogous. Take $u\in \F_{loc}\cup\F^\Lambda_{loc}$, by Lemma \ref{lem:cutoff} in the appendix, $\eta u\in\F_B$, therefore we can apply \eqref{eq:localsobolev} and get 
\[
\|\eta u\|_{\rho,B}^2\lesssim C_S^B |B|^\frac{2-d}{d} \E(\eta u,\eta u).
\]
To get \eqref{eq:localsobolev_cutoff} we compute $\nabla(\eta u) = u \nabla \eta+\eta\nabla u$ and we easily estimate
\begin{align*}
\E(\eta u,\eta u) &= \int_{\R^d} \langle a \nabla(\eta u),\nabla(\eta u)\rangle dx\\&\leq 2\int_{\R^d} \langle a \nabla u,\nabla u\rangle \eta^2  dx + 2\int_{\R^d} \langle a \nabla \eta,\nabla \eta\rangle |u|^2  dx 
\\&\leq2\E_\eta(u,u)+2\|\nabla\eta\|_\infty^2\|1_B u\|_{2,\Lambda}^2.
\end{align*}
Concatenating the two inequalities and averaging over $B$ we get the result.
\end{proof}

\subsection{Nash inequalities}

Local Nash inequalities follow as an easy corollary of the Sobolev's inequalities \eqref{eq:localsobolev} and \eqref{eq:localweightedsobolev}. 

\begin{proposition}[Nash inequality]
Let $B\subset\R^d$ be a ball. Then for all $u\in \F_B$ we have
\begin{equation}\label{ineq:nash}
\|u\|_{2,B}^{2+\frac{2}{\mu}}\lesssim C_S^B  |B|^{\frac{2-d}{d}}\E(u,u) \|u\|_{1,B}^{\frac{2}{\mu}},
\end{equation}
where $\mu\defeq (\frac{2}{d}-\frac{1}{q})^{-1}>0$, and
\begin{equation}\label{ineq:nash_weighted}
\|u\|_{2,\Lambda,B}^{2+\frac{2}{\gamma}}\lesssim C_S^{B,\Lambda}  |B|^{\frac{2-d}{d}}\E(u,u) \|u\|_{1,\Lambda,B}^{\frac{2}{\gamma}},
\end{equation}
where $\gamma\defeq \frac{p-1}{p}\, (\frac{2}{d}-\frac{1}{p}-\frac{1}{q})^{-1}$.
\end{proposition}
\begin{proof}
We prove only \eqref{ineq:nash} being the other completely analogous. By H\"older's inequality
\[
\|u\|_{2,B}\leq \|u\|_{\rho,B}^{\theta}\|u\|_{1,B}^{1-\theta}
\]
with $\theta\in(0,1)$ and
\[
\frac{1}{2}=(1-\theta)+\frac{\theta}{\rho}.
\]
Now solve for $\theta$, use \eqref{eq:localsobolev} to estimate $\|u\|_{\rho,B}$ and the result is obtained.
\end{proof}

Note that the condition $1/p+1/q<2/d$ is important to have $\mu$ and $\gamma$ positive, in particular $\gamma\geq d/2$, with the equality holding if $p=q=\infty$. It is well known that Nash inequality \eqref{ineq:nash_weighted}  for the Dirichlet form $(\E,\F_B^\Lambda)$ implies the ultracontractivity of the semigroup $P^B_t$ associated to $\E$ on $L^2(B,\Lambda dx)$, in particular there exists a density $p^B_t(x,y)$ with respect to $\Lambda(x)dx$ which satisfies
\[
\sup_{x,y\in B} p^B_t(x,y) \lesssim  t^{-\gamma}C_S^B |B|^{\frac{2}{d}-\frac{1}{\gamma}},
\]
where it is once more worthy to notice that $2/d-1/\gamma\geq 0$, with the equality holding for the non-degenerate situation.

Furthermore, we have just seen that $P_t$ is locally ultracontractive, being $P_t^B$ ultracontractive for all balls $B\subset\R^d$. It follows by Theorem 2.12 of \cite{grigor'yan2012} that $P_t$ admits a symmetric transition kernel $p_t(x,y)$ on $(0,\infty)\times\R^d\times\R^d$ with respect to $\Lambda(x)dx$.
\subsection{Poincar\'e inequalities}

Let $B\subset \R^d$ be a ball. Given a weight $\theta:B\to [0,\infty]$, we denote by
\[
(u)^{\theta}_B\defeq \int_B u\,\theta dx\Big/ \int_B \theta dx,
\]
if $\theta\equiv 1$ we simply write $(u)_B$. Moreover, for $u\in\F_{loc}$ we denote
\[
\E_B(u,u)\defeq \int_B a\nabla u\cdot \nabla u\, dx.
\]
\begin{proposition}[Poincar\'e inequalities] Let $B\subset \R^d$ be a ball. If $u\in \F_{loc}$, then
\begin{equation}\label{ineq:Poincare}
\|u-(u)_B\|^2_{2,B}\lesssim C_P^B |B|^\frac{2-d}{d} \E_B(u,u),
\end{equation}
being $C_P^B\defeq \|\lambda^{-1}\|_{d/2,B}$, and
\begin{equation}\label{ineq:Poincare_weighted}
\|u-(u)^\Lambda_B\|^2_{2,B,\Lambda}\lesssim C^{B,\Lambda}_{P}|B|^\frac{2-d}{d}\E_B(u,u),
\end{equation}
being $C_P^{B,\Lambda}\defeq \|\Lambda\|_{\bar{p},B}\|\lambda^{-1}\|_{\bar{q},B}$ with $\bar{p},\bar{q}\in[1,\infty]$ such that $1/\bar{p}+1/\bar{q}=2/d$.
\end{proposition}
\begin{proof} For \eqref{ineq:Poincare} use H\"older's inequality for the standard Sobolev inequality \cite[Theorem 1.5.2]{saloff2002aspects}. We now prove \eqref{ineq:Poincare_weighted} for $u\in C^\infty(B)$, the final result can be obtained by approximation. As first remark, notice that
\begin{align*}
\|u-(u)^\Lambda_B&\|^2_{2,B,\Lambda} = \inf_{a\in\R}\|u-a\|^2_{2,B,\Lambda}\\ &\leq \|\Lambda\|_{\bar{p},B}\inf_{a\in\R}\|u-a\|^2_{2\bar{p}^*,B}\leq \|\Lambda\|_{\bar{p},B}\|u-(u)_B\|^2_{2\bar{p}^*,B}.
\end{align*}
We have by Theorem 1.5.2  in \cite{saloff2002aspects}.
\[
\|u-(u)_B\|^2_{2\bar{p}^*,B}\lesssim |B|^{\frac{2}{d}}\|\nabla u\|^2_{\beta,B}\leq \|\lambda^{-1}\|_{\bar{q},B}|B|^{\frac{2-d}{d}}\E_B(u,u).
\]
where $\beta$ is such that $2 \bar{p}^* d/(d+2\bar{p}^*) = \beta = 2\bar{q}/(\bar{q}+1)$, which is true whenever $1/\bar{p}+1/\bar{q}=2/d$.
Concatenating the two inequalities leads to the result.
\end{proof}

In order to get mean value inequalities for the logarithm of caloric functions and, given that, the parabolic Harnack inequality, we will need a Poincar\'e inequality with a radial cutoff. The cutoff function $\eta:\R^d\to[0,\infty)$ is supported in a ball $B=B(x_0,r)$, is a radial function, $\eta(x)\defeq \Phi(|x-x_0|/r)$ where $\Phi$ is some non-increasing, non-negative c\`adl\`ag function non identically zero on $(r/2,r]$.

\begin{proposition}[Poincar\'e inequalities with radial cutoff] Let $B\subset \R^d$ be a ball of radius $r>0$ and center $x_0$ and let $\eta$ be a cutoff as above. If $u\in \F_{loc}$, then
\begin{equation}\label{ineq:Poincare_eta}
\|u-(u)^{\eta^2}_B\|_{2,B,\eta^2}\lesssim M^B C_P^B |B|^\frac{2-d}{d}\E_\eta(u,u)
\end{equation}
where $M^B = \Phi(0)/\Phi(1/2)$, and
\begin{equation}\label{ineq:Poincare_etaweighted}
\|u-(u)^{\Lambda\eta^2}_B\|_{2,B,\Lambda\eta^2}\lesssim M^{B,\Lambda} C_P^{B,\Lambda} |B|^\frac{2-d}{d}\E_\eta(u,u),
\end{equation}
where $M^{B,\Lambda}\defeq M^B \|\Lambda\|_{1,B}/\|\Lambda\|_{1,B/2}$.
\end{proposition}
\begin{proof}
We give the proof only for \eqref{ineq:Poincare_etaweighted} being \eqref{ineq:Poincare_eta} similar.
We apply Theorem 1 in \cite{DydaKassmann}. Accordingly we define a functional $F(u,s):L^2(\R^d,\Lambda dx)\times(r/2,r]\to[0,\infty]$ by
\[
F(u,s)\equiv C_P^{B_s,\Lambda} |B_s|^\frac{2}{d}\int_{B_s} a\nabla u\cdot\nabla u dx.
\]
for $u\in \F^\Lambda$, and $F(u,s)=\infty$ otherwise, being $B_s$ the ball of center $x_0$ and radius $s\in(r/2,r]$.

Such functional satisfies $F(u+a,s)=F(u,s)$ for all $a\in\R$ and $u\in L^2(\R^d,\Lambda dx)$, moreover
\[
\|u-(u)_{B_s}\|_{2,B_s,\Lambda}^2\lesssim |B_s|^{-1}F(u,s)
\]
for every $s\in(r/2,r]$ and $u\in \F^\Lambda$ by the Poincar\'e inequality \eqref{ineq:Poincare_weighted}. It follows from Theorem 1 in \cite{DydaKassmann} that for $u\in\F^\Lambda$ there exists $M>0$, explicitly given by $(\|\Lambda\|_{1,B} \Phi(0))/(\|\Lambda\|_{1,B/2}\Phi(1/2))$, such that
\begin{align*}
\|u-(u)^{\Lambda\eta^2}_{B}\|_{2,B,\Lambda\eta^2}&\lesssim M |B|^{-1} \int_{r/2}^{r} F(u,s)\nu(ds)\\
&\lesssim M C_P^{B,\Lambda}  |B|^\frac{2-d}{d}\int_{r/2}^{r} \int_B a\nabla u\cdot\nabla u 1_{B_s} dx\,\gamma(ds)\\
&=  M C_P^{B,\Lambda} |B|^\frac{2-d}{d} \int_B \eta^2 a\nabla u\cdot\nabla u dx.
\end{align*}
Here $\gamma(ds)$ is a non-zero positive $\sigma$-finite Borel measure on $(r/2,r]$ such that
\[
\eta^2(x)=\int_{r/2}^{r}1_{B_s}(x)\,\nu(ds)
\]
as in \cite{DydaKassmann}. 
Of course such an inequality is local and we can extend it for $u\in \F_{loc}$.
\end{proof}

\subsection{Remark on the constants}\label{subsec:constants} In this section about inequalities, we have introduced different constants, among them $C^{B,\Lambda}_{S}$, $C^{B,\Lambda}_{P}$ and $M^{B,\Lambda}$. Observe that they all strongly depend on the ball $B$, on the radius and the center as well. Assumption \ref{ass:b.2} helps us to control the behavior of the constants as the radius of the ball increases. For example, let us have a look at $C^{B,\Lambda}_{S}$. We have by \ref{ass:b.2}
\[
\limsup_{r\to\infty} C_S^{B(0,r),\Lambda} = \limsup_{r\to\infty} \|\lambda^{-1}\|_{q,B(0,r)}\|\Lambda\|_{p,B(0,r)}^{2p^*/\rho} =: C_S^{*,\Lambda}<\infty
\] 
It is not difficult to prove that the limit doesn't change if we consider balls centered at any $x\in\R^d$. Indeed it is easy to see that for $r>|x|$
\[
\biggl(\frac{r-|x|}{r}\biggr)^{\frac{2}{d}-\frac{1}{\gamma}}C_S^{B(0,r-|x|),\Lambda}\leq C_S^{B(x,r),\Lambda}\leq \biggl(\frac{r+|x|}{r}\biggr)^{\frac{2}{d}-\frac{1}{\gamma}} C_S^{B(0,r+|x|),\Lambda}
\]

We must be a bit careful with $M^{B,\Lambda}=M^{B} \|\Lambda\|_{1,B}/\|\Lambda\|_{1,B/2}$, since $\|\Lambda\|_{1,B/2}$ appears in the denominator. Still, it easy to see that
\[
\limsup_{r\to\infty}\|\Lambda\|^{-1}_{1,B(0,r)} \leq \limsup_{r\to\infty}\|\lambda^{-1}\|_{1,B(0,r)}\leq\limsup_{r\to\infty}\|\lambda^{-1}\|_{q,B(0,r)}<\infty,
\]
from which it follows that
$
\liminf_{r\to\infty}\|\Lambda\|_{1,B(0,r)}>0$
and in particular
\[
\limsup_{r\to\infty} M^{B(0,r),\Lambda} =: M^{*,\Lambda} <\infty.
\]
Again, it is not hard to show that the limit does not change if we consider balls centered at any $x\in\R^d$.

\begin{remark}Assumption \ref{ass:b.2} implies in particular that given $\delta>0$ for each $x\in\R^d$ there exists $s(x,\delta)$ such that 
\[
C_S^{B(x,r),\Lambda}\leq (1+\delta)C_S^{*,\Lambda},\quad M^{B(x,r),\Lambda}\leq (1+\delta)M^{*,\Lambda},\quad C_P^{B(x,r),\Lambda}\leq (1+\delta)C_P^{*,\Lambda}
\]
for all $r> s(x,\delta)$.
\end{remark}
\section{Estimates for caloric functions}

\subsection{Mean value inequalities for subcaloric functions}
To avoid the same type of technical problems which we faced in \cite[section 2.3]{chiarinideuschel}, we shall assume that our positive subcaloric functions $u$ are locally bounded. It turns out that any positive subcaloric function is locally bounded; this can be proved repeating the argument below with some additional technicalities similar to what we did in the proof of \cite{chiarinideuschel}[Proposition 2.4].

\begin{proposition}\label{prop:premoser} Consider $I=(t_1,t_2)\subset \R$ and a ball $B\subset\R^d$. Let $u$ be a locally bounded positive subcaloric function in $Q=I\times B$. Take cutoffs $\eta\in C_0^\infty(B)$, $0\leq \eta\leq 1$ and  $\zeta : \R \to [0,1]$, $\zeta \equiv 0$ on $(-\infty,t_1]$. Set $\nu = 2-2p^*/\rho$. Then for all $\alpha\geq 1$
	\begin{equation}\label{ineq:moserstep}
	\|\zeta \eta^2 u^{2\alpha}\|^{\nu}_{\nu,I\times B,\Lambda}\lesssim  C_S^{B,\Lambda}\frac{|B|^{\frac{2}{d}}}{|I|^{1-\nu}} \Bigl[\alpha(\|\zeta'\|_\infty+\|\nabla \eta\|^2_{\infty})\Bigr]^{\nu} \| u^{2\alpha}\|_{1,I\times B,\Lambda}^{\nu}.
	\end{equation}
\end{proposition}
\begin{proof} Since $u_t>0$ is locally bounded, the power function $F:
	\R\to \R$ defined by $F(x)= |x|^{2 \alpha}$ with $\alpha\geq 1$ satisfies the assumptions of Lemma \ref{lemma:parabolic}. Thus, for $\eta\in C^\infty_0(B)$ as above we have
\begin{equation}\label{eq:ca}
\frac{d}{dt} ( u_t^{2\alpha},\eta^2 )_{\Lambda} + 2\alpha\,\E(u_t,u_t^{2\alpha-1}\eta^2)\leq0,\quad t\in I.
\end{equation}
We can estimate
\begin{align*}
\E(u_t,u_t^{2\alpha-1}\eta^2) &= 2\int \eta u_t^{2\alpha-1} \langle a \nabla u_t,\nabla \eta \rangle\,dx+ (2\alpha-1)\int \eta^2 u_t^{2\alpha-2} \langle a \nabla u_t,\nabla u_t \rangle\,dx\\
&\geq \frac{2\alpha-1}{\alpha^2}  \,\E_{\eta}(u^\alpha,u^\alpha )-\frac{2\|\nabla \eta\|_\infty}{\alpha}\,\E_{\eta}(u^\alpha,u^\alpha )^{1/2} \|1_B u^{2\alpha}\|^{1/2}_{1,\Lambda},
\end{align*}
by means of Young's inequality $2ab\leq (\epsilon a^2+b^2/\epsilon)$ with $a=\E_{\eta}(u^\alpha,u^\alpha )^{1/2}$ and $b=\|\nabla \eta\|_\infty\|1_B u^{2\alpha}\|^{1/2}_{1,\Lambda}$ and for $\epsilon = 1/2\alpha$, we get exploiting that $\alpha\geq 1$
\[
\E(u_t,u_t^{2\alpha-1}\eta^2)\geq (1/2\alpha) \E_{\eta}(u^\alpha,u^\alpha ) - 2 \|\nabla \eta\|_\infty^2\|1_B u^{2\alpha}\|_{1,\Lambda}.
\]
Going back to \eqref{eq:ca} we have
\[
\frac{d}{dt} \| (u_t^\alpha\eta)^2 \|_{1,\Lambda}+\E_\eta(u_t^\alpha,u_t^\alpha)\leq 4\alpha \|\nabla \eta\|_\infty^2\|1_B u^{2\alpha}\|_{1,\Lambda}
\]
We now take a smooth cutoff in time $\zeta : \R \to [0,1]$, $\zeta \equiv 0$ on $(-\infty,t_1]$, where $I=(t_1,t_2)$. We multiply the inequality above by $\zeta$ and integrate in time. This yields
\begin{equation*}
\zeta(t)\| (u_t^\alpha\eta)^2 \|_{1,\Lambda}+\int_{t_1}^t \zeta(s)\E_\eta(u_s^\alpha,u_s^\alpha)\,ds\\ \leq 4\alpha \Bigl[\|\zeta'\|_\infty+\|\nabla \eta\|^2_{\infty}\Bigr]\int_{t_1}^{t}\| 1_B u_s^{2\alpha} \|_{1,\Lambda}\,ds,
\end{equation*}
after averaging and taking the supremum for $t\in I$ we get
\begin{equation}\label{ineq:Moser_step}
\sup_{t\in I}\zeta(t)\| (\eta u_t^\alpha)^2 \|_{1,B,\Lambda}+\int_I \zeta(s)\frac{\E_\eta(u_s^\alpha,u_s^\alpha)}{|B|}\,ds \lesssim \alpha \Bigl[\|\zeta'\|_\infty+\|\nabla \eta\|^2_{\infty}\Bigr]\int_I\| u_s^{2\alpha} \|_{1,B,\Lambda}\,ds.
\end{equation}

We use \eqref{ineq:Moser_step} together with \eqref{eq:localweightedsobolev_cutoff} to get \eqref{ineq:moserstep}. Observe that $\nu = 2-2p^*/\rho$ is greater than one, since $\rho>2p^*$ by the condition $1/p+1/q<2/d$.
	Using H\"older's inequality and some easy manipulation 
	\[
	\|(\eta u_s^{\alpha})^2\|^\nu_{\nu,B,\Lambda}\leq \|\eta u_s^{\alpha}\|^{2}_{\rho/p^*,B,\Lambda}\|(\eta u_s^{\alpha})^2\|^{\nu-1}_{1,B,\Lambda},
	\]
	we can then integrate this inequality against $\zeta(s)^\nu$ over $I$ and obtain
	\[
	\frac{1}{|I|}	\int_I\zeta(s)^\nu\|\eta^2 u_s^{2\alpha}\|^\nu_{\nu,B,\Lambda}\,ds\leq\Bigl(\sup_{s\in I}\zeta(s)\|(\eta u_s^{\alpha})^2\|_{1,B,\Lambda}\Bigr)^{\nu-1} \frac{1}{|I|}\int_I \zeta(s)\|\eta u_s^{\alpha}\|^{2}_{\rho/p^*,B,\Lambda}\,ds.
	\]
	In view of the Sobolev inequality \eqref{eq:localweightedsobolev_cutoff} we have
	\[
	\|\eta u_s^\alpha\|_{\rho/p^*,B,\Lambda}^2\lesssim C_{S}^{B,\Lambda} |B|^{\frac{2}{d}} \Bigl[\frac{\E_\eta(u_s^\alpha,u_s^\alpha)}{|B|}+\|\nabla \eta\|_\infty^2\| u_s^{2\alpha}\|_{1,B,\Lambda}\Bigr].
	\]
	by \eqref{ineq:Moser_step} we can bound each of the two factors. We end up with the following iterative step
	\begin{equation*}
	\|\zeta \eta^2 u^{2\alpha}\|^{\nu}_{\nu,I\times B,\Lambda}\lesssim  C_S^{B,\Lambda}\frac{|B|^{\frac{2}{d}}}{|I|^{1-\nu}} \Bigl[\alpha(\|\zeta'\|_\infty+\|\nabla \eta\|^2_{\infty})\Bigr]^{\nu} \| u^{2\alpha}\|_{1,I\times B,\Lambda}^{\nu},
	\end{equation*}
	which is what we wanted to prove.
\end{proof} 
	
	The main idea is to use Moser's iteration technique on a sequence of parabolic balls; Proposition \ref{prop:premoser} with suitable choice of the cutoffs and of the parameter $\alpha$ is the iteration step. Fix a parameter $\tau>0$, let $x\in \R^d$, and $r>0$. Consider also a parameter $\delta\in (0,1)$. Then we define the parabolic balls
	\begin{align*}
	Q(\tau,x,s,r)=Q&=(s-\tau r^2,s)\times B(x,r)\\
	Q_{\delta}&= (s-\delta\tau r^2,s)\times B(x,\delta r)
	\end{align*}
	Clearly $Q_\delta\subset Q$ for all $\delta\in(0,1)$.
	
	\begin{theorem}\label{thm:mvi} Fix $\tau>0$ and let $1/2\leq\sigma'<\sigma\leq 1$. Assume that $1/p+1/q<2/d$ and let $u_t$ be a positive subcaloric function on $Q=Q(\tau,x,s,r)$. Then there exists a positive constant $C_1:=C_1(d,p,q)$ such that 
\begin{equation}\label{ineq:max}
					\sup_{Q_{\sigma'}} u(t,z)\leq C_1  (C_S^{B,\Lambda})^{\frac{1}{2\nu-2}} \tau^{\frac{1}{2}}\Biggl[ \frac{1+\tau^{-1}}{ (\sigma-\sigma')^2}\Biggr]^{\frac{\nu}{2\nu-2}} \| u\|_{2,Q_{\sigma},\Lambda},					\end{equation}
where $\nu = 2-2p^*/\rho$.
	\end{theorem}
	\begin{proof} We want to apply \eqref{ineq:moserstep} with a suitable sequence of cutoffs $\eta_k$ and $\zeta_k$. Set
	\[
	\sigma_k = \sigma'+2^{-k}(\sigma-\sigma'),\quad\delta_{k}= 2^{-k-1}(\sigma-\sigma')
	\]
	then $\sigma_k-\sigma_{k+1}= \delta_k$, then consider a cutoff $\eta_k:\R^d\to [0,1]$, such that $\supp \eta_k\subset B(\sigma_k r)$ and $\eta_k\equiv 1$ on $B(\sigma_{k+1}r)$, moreover assume that $\|\nabla \eta\|_{\infty} \leq 2/(r \delta_k)$. Take also a cutoff in time $\zeta:\R\to[0,1]$, $\zeta_k\equiv 1$ on $I_{\sigma_{k+1}}=(s-\sigma_{k+1}\tau r^2,s)$, $\zeta_k\equiv0$ on $(-\infty,s-\sigma_{k}\tau r^2)$ and $\|\zeta'\|_{\infty}\leq 2/( r^2\tau \delta_k)$. Let $\alpha_k = \nu^k$ with $\nu=2-2p^*/\rho$ as above. Then, an application of \eqref{ineq:moserstep} and using the fact that $\alpha_{k+1}=\nu\alpha_k$ yields
	\[
	\|u\|_{2\alpha_{k+1},Q_{\sigma_{k+1}},\Lambda}\leq \Biggl\{c(d) C_S^{B,\Lambda} \tau^{\nu-1}\Bigl[\frac{\alpha_k(1+\tau^{-1}) 2^{2k}}{ (\sigma-\sigma')^2}\Bigr]^\nu\Biggr\}^{\frac{1}{2\alpha_{k+1}}}   \| u\|_{2\alpha_k,Q_{\sigma_k},\Lambda}.
	\]
	where we used the fact that $\sigma_k/\sigma_{k+1}<2$, and that $\sigma_k\in[1/2,1]$.
	This is the starting point for Moser's iteration. Iterating the inequality from $i=0$ up to $k$ we get at the price of a constant $C_1>0$ which depends on $p,q$ and the dimension
	\[
		\|u\|_{2\alpha_{k},Q_{\sigma_{k}},\Lambda}\leq C_1  (C_S^{B,\Lambda})^{\frac{1}{2\nu-2}} \tau^{\frac{1}{2}}\Biggl[ \frac{1+\tau^{-1}}{ (\sigma-\sigma')^2}\Biggr]^{\frac{\nu}{2\nu-2}} \| u\|_{2,Q_{\sigma},\Lambda}.
	\]
	where we exploited the fact that $\sum_{i=0}^{\infty} 1/\alpha_i=\nu/(\nu-1)$ and that $\sum_{i=0}^{\infty} k/\alpha_i<\infty$. From the inequality above we easily get, taking $C_1$ larger if needed, 
	\[
	\|u\|_{2\alpha_{k},Q_{\sigma'},\Lambda}\leq C_1  (C_S^{B,\Lambda})^{\frac{1}{2\nu-2}} \tau^{\frac{1}{2}}\Biggl[ \frac{1+\tau^{-1}}{ (\sigma-\sigma')^2}\Biggr]^{\frac{\nu}{2\nu-2}} \| u\|_{2,Q_{\sigma},\Lambda}.
	\]
			and taking the limit as $k\to \infty$ gives the result
	\[
	\sup_{Q_{\sigma'}} u(t,z)\leq C_1  (C_S^{B,\Lambda})^{\frac{1}{2\nu-2}} \tau^{\frac{1}{2}}\Biggl[ \frac{1+\tau^{-1}}{ (\sigma-\sigma')^2}\Biggr]^{\frac{\nu}{2\nu-2}} \| u\|_{2,Q_{\sigma},\Lambda}.
	\]
	\end{proof}

	\begin{corollary}
	\label{thm:mvi_cor} Fix $\tau>0$ and let $1/2\leq\sigma'<\sigma\leq 1$. Assume that $1/p+1/q<2/d$ and let $u$ be a subcaloric function in $Q=Q(\tau,x,s,r)$. Then there exists  a positive constant $C_2:=C_2(q,p,d)$ which depends only on the dimension and on $p,q$ such that for all $\alpha>0$
	\begin{equation}\label{cor:uplus}
							\sup_{Q_{\sigma'}} u(t,z)\lesssim C_2 2^{\frac{2}{\alpha^2}\frac{\nu}{\nu-1}} (C_S^{B,\Lambda})^{\frac{1}{\alpha\nu-\alpha}}\tau^{\frac{1}{\alpha}} \Biggl[ \frac{1+\tau^{-1}}{ (\sigma-\sigma')^2}\Biggr]^{\frac{\nu}{\alpha\nu-\alpha}} \| u\|_{\alpha,Q_{\sigma},\Lambda},
	\end{equation}

	\end{corollary} 
	\begin{proof}To prove \eqref{cor:uplus} one can follow the same approach in \cite{saloff2002aspects}[Theorem 2.2.3] with the only difference that we will consider parabolic balls $Q_\sigma$ instead of balls. Observe that for $\alpha>2$ this is just an application of Jensen's inequality.
	\end{proof}
	Observe that \eqref{cor:uplus} is not good for the application of Bombieri-Giusti's lemma (\ref{lemma:BombieriGiusti}) since $2^{\frac{2}{\alpha}\frac{\nu}{\nu-1}}$ is exploding as $\alpha$ approaches zero. To get rid of this problem we develop in the next section the same type of inequalities for supercaloric functions.

	Theorem \ref{thm:mvi} can  be also applied to obtain a global on-diagonal heat kernel upper bound, as it is done in the next proposition.

	\begin{proposition} Let $f\in L^2(\R^d,\Lambda dx)$, and assume that \ref{ass:b.1} and \ref{ass:b.2} are satisfied, then there exists a constant $C_3=C_3(q,p,d,C_S^{*,\Lambda})>0$ such that for all $x\in \R^d$ and $t>0$ the following inequality holds
	\[
	 P_t f(x)\leq C_3 t^{-\gamma}(s(0,1) + |x| + \sqrt{t})^{\gamma-d/2} \int_{\R^d}(s(0,1) + |y| + \sqrt{t})^{\gamma-d/2} |f(y)|\Lambda(y)dy.
	\]
	where $\gamma$ was defined in \ref{ineq:nash_weighted} and $s(x,\delta)$ was defined in Section \ref{subsec:constants}.
	\end{proposition}
	\begin{proof} Assume that $\tau\in (0,2]$, $x=0$ and $r>0$, $s=\tau r^2$, $\sigma=1$ and $\sigma'=1/2$. It follows that
	\[
	Q_{1} = (0,\tau r^2)\times B(0,r),\quad Q_{1/2} = \tau r^2(1/2 ,1)\times B(0,r/2).
	\] 
	We chose $r = s(0,1) + 2|z| + \sqrt{t}$ where $s(0,1)$ was defined in Section \ref{subsec:constants}. In this way $C_S^{B,\Lambda}\leq 2 C_S^{*,\Lambda}$ and we can read inequality \eqref{ineq:max} for $u(s,z):= P_s f(z)$ as follows
	\[
	\sup_{Q_{1/2}} P_s f(z)\leq c (C_S^{*,\Lambda})^{\gamma/2} \frac {\tau^{-\gamma/2}}{r^{d/2}}  \| f\|_{2,\Lambda},
	\]
	with $c = c(p,q,d)$ changing throughout the proof.
	By definition of $r$ we find $\tau\in(0,2]$ such that $3/4\tau r^2 = t$, and in particular $(t,z)\in Q_{1/2}$. This gives
	\[
	 P_t f(z) \leq c t^{-\gamma/2} (s(0,1) + |z| + \sqrt{t})^{\gamma-d/2} \| f\|_{2,\Lambda}.
	\]
	and this holds for all $z\in \R^d$ and $t>0$. Set $b_t(z)=(s(0,1) + |z| + \sqrt{t})^{\gamma-d/2}$.
	It follows that
	\[
	 \|b_t^{-1} P_t f\|_{\infty} \leq c t^{-\gamma/2} \| f\|_{2,\Lambda},
	\]
	from which we deduce that 
	 $\|b_t^{-1} P_t \|_{2\to\infty} \leq c t^{-\gamma/2}$. And by duality we get $\|P_t b_t^{-1}\|_{1\to 2} \leq c t^{-\gamma/2}$. Hence
	 \[
	 \|P_t f\|_{2,\Lambda}\leq c t^{-\gamma/2}\|b_t f\|_{1,\Lambda}
	 \]
	Now it is left to use the semigroup property and standard techniques to finally get the bound.
  	\end{proof}

	It is now standard to get global on-diagonal estimates for the kernel $p_t(x,y)$ of the semigroup $P_t$ associated to $(\E,\F^\Lambda)$ on $L^2(\R^d,\Lambda dx)$. Namely we obtain that for almost all $x,y\in \R^d$ and for all $t>0$
\begin{equation}\label{ineq:ondiagonal}
p_t(x,y)\leq C_3 t^{-\gamma}(s(0,1) + |x| + \sqrt{t})^{\gamma-d/2} (s(0,1) + |y| + \sqrt{t})^{\gamma-d/2}.
\end{equation}
	\subsection{Mean value inequalities for supercaloric functions}

	\begin{theorem}\label{thm:mvisuper} Fix $\tau>0$ and let $1/2\leq\sigma'<\sigma\leq 1$. Assume that $1/p+1/q<2/d$ and let $u_t$ be a positive supercaloric function of on $Q=Q(\tau,x,s,r)$. Then there exists a positive constant $C_4:=C_4(p,q,d)$ which depends only on the dimension and on $p,q$ such that for all $\alpha\in(0,\infty)$
	\begin{equation}
						\sup_{Q_{\sigma'}} u(t,z)^{-\alpha}\leq   C_4(C_S^{B,\Lambda})^{\frac{1}{\nu-1}}\tau \Biggl[ \frac{1+\tau^{-1}}{ (\sigma-\sigma')^2}\Biggr]^{\frac{\nu}{\nu-1}} \| u^{-1}\|^\alpha_{\alpha,Q_{\sigma},\Lambda}.
	\end{equation}
	where $\nu = 2-2p^*/\rho$.
	\end{theorem}
	\begin{proof} We can always assume that $u>\epsilon$ by considering the supersolution $u+\epsilon$ and then sending $\epsilon$ to zero at the end of the argument. Applying Lemma \ref{lemma:parabolic} with the function $F(x):= -|x|^{-\beta}$ and $\beta>0$ we get
	\[
	-\frac{d}{dt}\|\eta^2 u_t^{-\beta}\|_{1,\Lambda} +\beta\, \E(u_t^{-\beta-1}\eta^2,u_t) \geq 0
	\]
	which after some manipulation gives
	\[
		-\frac{d}{dt}\|\eta^2 u_t^{-\beta}\|_{1,\Lambda} -4\frac{\beta+1}{\beta}\, \E_{\eta^2}(u_t^{-\beta/2},u_t^{-\beta/2}) -4\int a \nabla\eta\cdot\nabla (u_t^{-\beta/2}) \eta u_t^{-\beta/2}dx\geq 0
	\]
	by means of Young's inequality $4ab\leq 3a^2+2b^2/3$ and using the simple fact that $(\beta+1)/\beta > 1$ we get after averaging
	\[
	\frac{d}{dt}\|\eta^2 u_t^{-\beta}\|_{1,B,\Lambda} + \frac{\E_{\eta^2}(u_t^{-\beta/2},u_t^{-\beta/2})}{|B|} \lesssim \|\nabla \eta\|^2_{\infty}\|u^{-\beta}\|_{1,B,\Lambda}
	\]
		We now integrate against a time cutoff $\zeta:\R\to[0,1]$ to obtain something similar to  \eqref{ineq:Moser_step}. Hence the same approach as in Proposition \ref{prop:premoser} applies and we get
		\[
			\|\zeta \eta^2 u^{-\beta}\|^{\nu}_{\nu,I\times B,\Lambda}\lesssim  C_S^{B,\Lambda}\frac{|B|^{\frac{2}{d}}}{|I|^{1-\nu}} \Bigl[\|\zeta'\|_\infty+\|\nabla \eta\|^2_{\infty}\Bigr]^{\nu} \| u^{-\beta}\|_{1,I\times B,\Lambda}^{\nu}.
		\]
		Moser's iteration technique with $\beta_k=\nu^k \alpha$ and $\alpha>0$ and the same argument of Theorem 
		\ref{thm:mvi} will finally give 
	\begin{equation*}
						\sup_{Q_{\sigma'}} u(t,z)^{-\alpha}\leq C_4 (C_S^{B,\Lambda})^{\frac{1}{\nu-1}}  \tau \Biggl[ \frac{1+\tau^{-1}}{ (\sigma-\sigma')^2}\Biggr]^{\frac{\nu}{\nu-1}} \| u^{-1}\|^\alpha_{\alpha,Q_{\sigma},\Lambda}.
	\end{equation*}
	\end{proof}

	We introduce the following parabolic ball. Given $x\in\R^d$, $r,\tau>0$ and $s\in\R$, $\delta\in(0,1)$, we note
	\[
	Q'_\delta = Q'_\delta(\tau,x,s,r) = (s-\tau r^2, s-(1-\delta)\tau r^2)\times B(x,\delta r).
	\]

	\begin{theorem}\label{thm:mvisuper_close} Fix $\tau>0$ and let $1/2\leq\sigma'<\sigma\leq 1$. Assume that $1/p+1/q<2/d$ and let $u$ be a positive supercaloric function on $Q=Q(\tau,x,s,r)$. Fix $0<\alpha_0<\nu$. Then there exists a positive constant $C_5:=C_5(q,p,d,\alpha_0)$ which depends only on the dimension, on $p,q$ and on $\alpha_0$ such that for all $0<\alpha< \alpha_0 \nu^{-1}$ we have
		\begin{equation}
  		 	\|u\|_{\alpha_0,Q'_{\sigma'},\Lambda}\leq  \Biggl\{ C_5 \tau(1+\tau^{-1})^\frac{\nu}{\nu-1}\biggl[\frac{1\vee C_S^{B,\Lambda} }{ (\sigma-\sigma')^2}\biggr]^\frac{\nu}{\nu-1} \Biggr\}^{(1+\nu)(1/\alpha-1/\alpha_0)} \| u\|_{\alpha,Q'_{\sigma},\Lambda}
		\end{equation}
		where $\nu = 2-2p^*/\rho$.
	\end{theorem}
	\begin{proof} Assume $u$ is supercaloric on $Q=I\times B$. Applying Lemma \ref{lemma:parabolic} with the function $F(x):= |x|^{\beta}$ with $\beta\in (0,1) $ we get
		\[
		\frac{d}{dt}\|\eta^2 u_t^{\beta}\|_{1,\Lambda} +\beta\, \E(u_t^{\beta-1}\eta^2,u_t) \geq 0
		\]
		which after some manipulation gives
		\[
			\frac{d}{dt}\|\eta^2 u_t^{\beta}\|_{1,\Lambda} +4\frac{\beta-1}{\beta}\, \E_{\eta}(u_t^{\beta/2},u_t^{\beta/2})+4\int a \nabla\eta\cdot\nabla (u_t^{\beta/2}) \eta u_t^{\beta/2}dx\geq 0
		\]
	Note that $(\beta-1)$ is negative. If we take $0<\beta<\alpha_0 \nu^{-1}$ then we have
	\[
	\frac{1-\beta}{\beta} > 1-\beta > 1-\alpha_0/\nu =:\epsilon,
	\]
 	this yields after Young's inequality 
 	\[
 	-\frac{d}{dt}\|\eta^2 u_t^{\beta}\|_{1,\Lambda}+\epsilon\, \E_{\eta}(u_t^{\beta/2},u_t^{\beta/2})\leq A\|\nabla\eta\|^2_{\infty}\|1_B u^\beta\|_{1,\Lambda},
 	\]
	where $A$ is a constant possibly depending on $q,p,\alpha_0$ and $d$ which will be changing throughout the proof.
 	Here we introduce a difference, the time cutoff $\zeta:\R\to[0,1]$, $\zeta \equiv 0$ on $(t_2,\infty]$, where $I=(t_1,t_2)$, is zero at the top of the time interval and not at the bottom. This gives after integrating, 
 	\[
 	\zeta(t) \|\eta^2 u_t^{\beta}\|_{1,\Lambda} + \int_t^{t_2}\zeta(s)\E_{\eta}(u_s^{\beta/2},u_s^{\beta/2})\,ds \leq A\, \Bigl[\|\zeta'\|_\infty+\|\nabla\eta\|^2_{\infty}\Bigr]\int_t^{t_2}\|1_B u^\beta\|_{1,\Lambda}
  	\]
  	which has the same flavor of \eqref{ineq:Moser_step}. Starting from this inequality, and repeating the argument we used for subcaloric functions, we end up with
  		\begin{equation}\label{ineq:calc}
  				\|\zeta \eta^2 u^{\beta}\|^{\nu}_{\nu,I\times B,\Lambda}\leq A\,  C_S^{B,\Lambda}\frac{|B|^{\frac{2}{d}}}{|I|^{1-\nu}} \Bigl[\|\zeta'\|_\infty+\|\nabla \eta\|^2_{\infty}\Bigr]^{\nu} \| u^{\beta}\|_{1,I\times B,\Lambda}^{\nu}.
  			\end{equation}
  		We use now the same iteration argument in Theorem 2.2.5 of \cite{saloff2002aspects}. Namely define $\alpha_i = \alpha_0 \nu^{-i}$. Fix $i\geq 0$ and apply \eqref{ineq:calc} with $\beta_j= \alpha_i \nu^{j-1}$, and $j=1,\dots,i$, then clearly $0<\beta_j<\alpha_0\nu^{-1}$, moreover set $\sigma_0 = \sigma$, $\sigma_j-\sigma_{j+1} = 2^{-j-1}(\sigma-\sigma')$, and fix the usual cutoffs $\eta_k:\R^d\to [0,1]$, such that $\supp \eta_k\subset B(\sigma_k r)$ and $\eta_k\equiv 1$ on $B(\sigma_{k+1}r)$ and $\|\nabla \eta\|_{\infty} \leq 2/(r \delta_k)$. Take also a cutoff in time $\zeta:\R\to[0,1]$, $\zeta_k\equiv 1$ on $I_{\sigma_{k+1}}=(s-\tau r^2,s-(1-\sigma_{k+1})\tau r^2)$, $\zeta_k\equiv0$ on $(s-(1-\sigma_{k})\tau r^2,\infty)$ and $\|\zeta'\|_{\infty}\leq 2/( r^2\tau \delta_k)$. Then we have for all $j=1,\dots,i$
  		\[
  		\|u^{\alpha_i\nu^{j}}\|_{1,Q'_{\sigma_j},\Lambda}\leq A \, C_S^{B,\Lambda} \tau^{\nu-1}\Bigl[\frac{(1+\tau^{-1}) 2^{2j}}{ (\sigma-\sigma')^2}\Bigr]^\nu  \| u^{\alpha_i\nu^{j-1}}\|_{1,Q'_{\sigma_{j-1}},\Lambda}^{\nu},
  		\]
  		which after an iteration from $j=1$ to $j=i$ gives
  		\[
  		  		\|u\|^{\alpha_0}_{\alpha_0,Q'_{\sigma_i},\Lambda}\leq    \Biggl\{A\,C_S^{B,\Lambda} \tau^{\nu-1}\Bigl[\frac{(1+\tau^{-1}) 2^{2(i-k)}}{ (\sigma-\sigma')^2}\Bigr]^\nu \Biggr\}^{\sum_{k=0}^{i-1} \nu^{k}} \| u^{\alpha_i}\|_{1,Q'_{\sigma},\Lambda}^{\nu^i}.
  		  		\]
  		Now observe that
  		\[
  		\sum_{k=0}^{i-1} (i-k) \nu^k\leq C(\nu)(\alpha_0/\alpha_i-1),\quad \sum_{k=0}^{i-1}\nu^k =\frac{\nu^i-1}{\nu-1}=\frac{\alpha_0/\alpha_i-1}{\nu-1}
  		\]
  		where $C(\nu)$ does not depend on $i$. This yields the following inequality
  		\[
  		  		  		\|u\|_{\alpha_0,Q'_{\sigma'},\Lambda}\leq \Biggl\{ A\, \tau(1+\tau^{-1})^\frac{\nu}{\nu-1}\biggl[\frac{1\vee C_S^{B,\Lambda} }{ (\sigma-\sigma')^2}\biggr]^\frac{\nu}{\nu-1} \Biggr\}^{1/\alpha_i-1/\alpha_0} \| u\|_{\alpha_i,Q'_{\sigma},\Lambda}
  		  		  		\]
  		 where the constant $A$ depends only on $\alpha_0, q, p$ and the dimension $d\geq 2$ and can be taken greater than one. Finally we extend the inequality for $\alpha\in(0,\alpha_0 \nu^{-1})$. Let $i\geq 2$ be an integer such that $\alpha_i\leq \alpha<\alpha_{i-1}$, then we have $1/\alpha_i-1/\alpha_0\leq (1+\nu)(1/\alpha-1/\alpha_0)$ and by means of Jensen's inequality we get
  		 \[
  		 	\|u\|_{\alpha_0,Q'_{\sigma'},\Lambda}\leq  \Biggl\{ A\, \tau(1+\tau^{-1})^\frac{\nu}{\nu-1}\biggl[\frac{1\vee C_S^{B,\Lambda} }{ (\sigma-\sigma')^2}\biggr]^\frac{\nu}{\nu-1} \Biggr\}^{(1+\nu)(1/\alpha-1/\alpha_0)} \| u\|_{\alpha,Q'_{\sigma},\Lambda},
  		 \]
  		 which is what we wanted to prove.
	\end{proof}

	\subsection{Mean value inequalities for $\log u_t$}
	
	In this section we get mean value inequalities for $\log u_t$ where $u_t$ is a positive supercaloric function on $Q= (s-\tau r^2, s)\times B(x,r)$, with $\tau>0$ fixed. We denote by $m^\Lambda:=\Lambda dx$ and by $\gamma^\Lambda := dt\times m^\Lambda$. 	
	\begin{theorem}\label{thm:loginequalities} Fix $\tau>0$ and $\kappa\in(0,1)$, $\delta\in[1/2,1)$. For any $s\in\R$ and $r>0$ and any positive supercaloric function $u$ on $Q= (s-\tau r^2, s)\times B(x,r)$, there exist a positive constant $C_6:=C_6(q,p,d,\delta)$  and a constant $k\defeq k(u,\kappa)>0$ such that
	\begin{equation}
	 \gamma^{\Lambda}\{(t,z)\in K^+\, |\,\log u_t<-\ell-k\}\leq C_6\, m^{\Lambda}(B) \Bigl[M^{B,\Lambda}|B|^\frac{2}{d} (C_{P}^{B,\Lambda}\vee \tau^2)\Bigr]\ell^{-1},
	\end{equation}
	and
	\begin{equation}
		 \gamma^{\Lambda}\{(t,z)\in K^-\, |\,\log u_t>\ell-k\}\leq C_6\, m^{\Lambda}(B) \Bigl[M^{B,\Lambda}|B|^\frac{2}{d} (C_{P}^{B,\Lambda}\vee \tau^2)\Bigr]\ell^{-1},
		\end{equation}
		where $K^+=(s-\kappa\tau r^2,s)\times B(x,\delta r)$ and $K^-=(s-\tau r^2,s-\kappa\tau r^2)\times B(x,\delta r)$.
	\end{theorem}
	\begin{proof}
	 We follow closely the strategy adopted in Theorem 5.4.1 of \cite{saloff2002aspects}. We can always assume $u_t\geq \epsilon$ and then send $\epsilon$ to zero in our estimates, since $u_t+\epsilon$ is still a supercaloric function. We denote as usual $B:=B(x,r)$. By Lemma \ref{lemma:parabolic}
	\begin{align}\label{eq:meanvalue}
	\frac{d}{dt}(\eta^2,-\log u_t)_{\Lambda} &\leq \E( u_t^{-1} \eta^2, u_t) = -\E_\eta(\log u_t, \log u_t)+2\int \langle a\nabla \eta,\nabla u_t\rangle \eta u_t^{-1} dx\\
	&\leq -\E_\eta(\log u_t, \log u_t)+2 \E_\eta(\log u_t,\log u_t)^{1/2} \|\nabla \eta\|_{\infty} \|1_B\|^{1/2}_{1,\Lambda}\notag\\
	&\leq-\frac{1}{2}\E_\eta(\log u_t, \log u_t)+2  m^\Lambda(B)\|\nabla \eta\|^2_{\infty} \notag
	\end{align}
	in the last inequality we exploit Young's inequality $2 ab \leq (1/2 a^2 + 2 b^2)$. The cutoff function $\eta$ must be on the form used in \eqref{ineq:Poincare_etaweighted}. We take
	\[ 
	\eta(z)\defeq (1-|x-z|/r)_+
	\] 
	where $x,r$ are the center and the radius of the ball $B$. We note
	\[
	w_t(z)\defeq -\log u_t(z),\quad W_t\defeq(w_t)^{\Lambda\eta^2}_B
	\]
	then \eqref{ineq:Poincare_etaweighted} reads
	\[
	\frac{|B|}{\|\eta^2 \Lambda\|_1}\|w_t-W_t\|^2_{2,B,\Lambda\eta^2} \lesssim M^{B,\Lambda}C_P^{B,\Lambda} |B|^\frac{2}{d}\frac{\E(w_t,w_t)}{2\|\eta^2\Lambda\|_1},
	\]
	rewriting \eqref{eq:meanvalue} we get
	\[
	\partial_t W_t+\frac{|B|}{\|\eta^2 \Lambda\|_1}\Bigl(M^{B,\Lambda} C_{P}^{B,\Lambda}|B|^\frac{2}{d}\Bigr)^{-1}\|w_t-W_t\|^2_{2, B,\Lambda \eta^2}\lesssim  \|\nabla \eta\|^2_{\infty} \frac{m^\Lambda(B)}{\|\eta^2 \Lambda\|_1},
	\]
	since $(1-\delta)^2 m^\Lambda( B(x,\delta r))\leq\|\eta^2\Lambda\|_1
	\leq m^\Lambda(B)$ and $\|\nabla \eta\|^2_\infty\lesssim |B|^{-\frac{2}{d}}$ we can write
	\begin{equation}\label{eq:ode}
	\partial_t W_t+\Bigl(m^\Lambda(B) M^{B,\Lambda}C_{P}^{B,\Lambda} |B|^\frac{2}{d}\Bigr)^{-1}\int_{\delta B}|w_t-W_t|^2\,\Lambda dx\leq  c\, M^{B,\Lambda}|B|^{-\frac{2}{d}}
	\end{equation}
	for some constant $c>0$ depending only on the dimension and $\delta$. Observe that we fixed $\delta\in [1/2,1)$ to stay away from the boundary of $B$, that for very large radius $M^{B,\Lambda}$ and $C_P^{B,\Lambda}$ are basically constants and that $m^\Lambda(B)$ is the volume of a ball in $L^2(\R^d,\Lambda dx)$; hence what we have above resembles closely what is given in \cite{saloff2002aspects}. Let us introduce the following auxiliary functions
	\[
		\bar{w}_t\defeq w_t-c\, M^{B,\Lambda}|B|^{-\frac{2}{d}}(t-s'),\quad \bar{W}_t\defeq W_t-c\, M^{B,\Lambda}|B|^{-\frac{2}{d}}(t-s'),
	\]
	where $s'= s - \kappa \tau r^2$. We can now rewrite \eqref{eq:ode} as 
\begin{equation}\label{eq:ode2}
	\partial_t \bar{W}_t+\Bigl(m^\Lambda(B) M^{B,\Lambda}C_{P}^{B,\Lambda} |B|^\frac{2}{d}\Bigr)^{-1}\int_{\delta B}|\bar{w}_t-\bar{W}_t|^2\,\Lambda dx\leq  0.
	\end{equation}
Now set $k(u,\kappa)\defeq \bar{W}_{s'}$ and define the two sets
\[
D^+_t(\ell)\defeq \{z\in  B(x,\delta r)\,|\,\bar{w}(t,z)>k+\ell\},
\]
\[
D^-_t(\ell)\defeq \{z\in B(x,\delta r)\,|\,\bar{w}(t,z)<k-\ell\}.
\]
since $\partial_t\bar{W}_t\leq 0$ we have that, for $t>s'$, $\bar{w}_t-\bar{W}_t> \ell+k(u)-\bar{W}_t\geq \ell$ on $D^+_t(\ell)$. Using this in \eqref{eq:ode2} we obtain
\begin{equation}
\partial_t \bar{W}_t+\Bigl(m^\Lambda(B)M^{B,\Lambda}C_{P}^{B,\Lambda} |B|^\frac{2}{d}\Bigr)^{-1}|\ell+k-\bar{W}_t|^2\,m^\Lambda(D^+_t(\ell)) \leq  0.
\end{equation}
or equivalently
\begin{equation}
-\Bigl(m^\Lambda(B) M^{B,\Lambda}C_{P}^{B,\Lambda} |B|^\frac{2}{d}\Bigr)\partial_t|\ell+k-\bar{W}_t|^{-1}\geq m^\Lambda(D^+_t(\ell)).
\end{equation}
Integrating from $s'$ to $s$ yields, for $\gamma^{\Lambda} = dt\times m^\Lambda$,
\[
 \gamma^{\Lambda}\{(t,z)\in K^+\, |\,\bar{w}(t,z)>k+\ell\}\leq m^{\Lambda}(B) \Bigl(M^{B,\Lambda}C_{P}^{B,\Lambda} |B|^\frac{2}{d}\Bigr)\ell^{-1}
\]
going back to $-\log u_t = \bar{w}_t+c\, M^{B,\Lambda}|B|^{-\frac{2}{d}}(t-s')$ we can rewrite
\[
 \gamma^{\Lambda}\{(t,z)\in K^+\, |\,\log u_t+c\, M^{B,\Lambda}|B|^{-\frac{2}{d}}(t-s')<-k-\ell\}\leq m^{\Lambda}(B) \Bigl(M^{B,\Lambda}C_{P}^{B,\Lambda} |B|^\frac{2}{d}\Bigr)\ell^{-1}.
\]
Finally,
\begin{align*}
 \gamma^{\Lambda}\{(t,z)\in K^+\, &|\,\log u_t<-k(u)-\ell\}\\
 &\leq \gamma^{\Lambda}\{(t,z)\in K^+\,|\,\log u_t+c\, M^{B,\Lambda}|B|^{-\frac{2}{d}}(t-s')<-k-\ell/2\}\\
 &+ \gamma^{\Lambda}\{(t,z)\in K^+\,|\,c\, M^{B,\Lambda}|B|^{-\frac{2}{d}}(t-s')>\ell/2\}\\
 &\lesssim m^{\Lambda}(B) \Bigl(M^{B,\Lambda}C_{P}^{B,\Lambda} |B|^\frac{2}{d}\Bigr)\ell^{-1} + m^{\Lambda}(B) \Bigl(\tau^2 M^{B,\Lambda} |B|^\frac{2}{d}\Bigr)\ell^{-1} \\
 &\lesssim m^{\Lambda}(B)\Bigl[M^{B,\Lambda}|B|^\frac{2}{d} (C_{P}^{B,\Lambda}\vee \tau^2)\Bigr]\ell^{-1}.
\end{align*}
where in the second but last step we used Markov's inequality and the fact that $\kappa<1$. Working with $D^{-}_t(\ell)$ and $K^-$ and using similar arguments proves the second inequality.
\end{proof}

\subsection{Parabolic Harnack's inequality}

We have all the tools to apply Lemma \ref{lemma:BombieriGiusti} effectively to a positive function $u$ which is caloric in the parabolic ball $Q(\tau,s,x,r) = (s-\tau r^2,s)\times B(x,r)$. This will finally gives us the parabolic Harnack's inequality. Fix $\delta\in(0,1)$ and $\tau >0$. For $x\in \R^d$ and $s\in\R$ and $r>0$ denote
\begin{align}\label{definitionsets}
Q_- &= (s-(3+\delta)\tau r^2/4,s-(3-\delta)\tau r^2/4)\times \delta B,\\
Q_-' &= (s-\tau r^2,s-(3-\delta)\tau r^2/4)\times \delta  B,\notag
\\
Q_+&=(s-(1+\delta)\tau r^2/4,s)\times \delta  B.\notag
\end{align}
Then we have the following.
\begin{theorem} Fix $\tau>0$ and $\delta\in[1/2,1)$. Fix $\alpha_0\in(0,\nu)$. Let $u$ be any positive caloric function on $Q= (s-\tau r^2,s)\times B(x,r)$. Then we have
\begin{equation}
\|u\|_{\alpha_0,Q_-',\Lambda}\lesssim C_7 \inf_{Q_+} u(t,z)
\end{equation}
where the constant $C_7$ depends increasingly on $C_S^{B,\Lambda},C_P^{B,\Lambda},M^{B,\Lambda}$, and on $\tau,p,q,\alpha_0, d,\delta$.
\end{theorem}
\begin{proof}
For the proof we follow closely \cite{saloff2002aspects}[Theorem 5.4.2].
Take $k:=k(u,\kappa)$ corresponding to $\kappa = 1/2$ in Theorem \ref{thm:loginequalities}. Set $v=e^{k}u$ and 
\[
U=(s-\tau r^2,s-1/2\tau r^2)\times B(x,r),\quad U_\sigma=(s-\tau r^2,s-(3-\sigma)\tau r^2/4)\times  B(x,\sigma r)
\]
By Theorem \ref{thm:mvisuper_close} it follows that
\[
  		 	\|v\|_{\alpha_0,U_{\sigma'},\Lambda}\leq  \Biggl\{ C_5 \tau(1+\tau^{-1})^\frac{\nu}{\nu-1}\biggl[\frac{1\vee C_S^{B,\Lambda} }{ (\sigma-\sigma')^2}\biggr]^\frac{\nu}{\nu-1} \Biggr\}^{(1+\nu)(1/\alpha-1/\alpha_0)} \| v\|_{\alpha,U_{\sigma},\Lambda}
\]
for all $1/2\leq\sigma'<\sigma\leq 1$ and all $\alpha\in(0,\alpha_0\nu^{-1})$, in particular notice that $\alpha_0\nu^{-1}>\alpha_0/2$ and that  $\alpha_0/2<\nu/2<1$ since $\nu\in(1,2)$. By Theorem \ref{thm:loginequalities} we have that
\[
		 \gamma^{\Lambda}\{(t,z)\in U\, |\,\log v>\ell\}\leq C_6 \, \gamma^{\Lambda}(U) \tau^{-1} \Bigl[M^{B,\Lambda} (C_{P}^{B,\Lambda}\vee \tau^2)\Bigr]\ell^{-1},
\]
Bombieri-Giusti's Lemma \ref{lemma:BombieriGiusti} is applicable and we obtain
\[
	\|e^\kappa u\|_{\alpha_0,Q_-',\Lambda} \lesssim C_{BG}^B
\]
where  $C_{BG}^B$ depends increasingly on $C_S^{B,\Lambda},C_P^{B,\Lambda},M^{B,\Lambda}$, and on $\tau,p,q,\alpha_0,d$. 
On the other hand we can now fix
\[
V=(s-1/2\tau r^2,s)\times B(x,r),\quad V_\sigma=(s-(1+\sigma)\tau r^2/4,s)\times B(x,\sigma r)
\] 
and apply Theorem \ref{thm:mvisuper} to $v=e^{-k} u^{-1}$ where $k$ is the same constant as above, this produces
\[
\sup_{V_{\sigma'}} v(t,z)\leq \Biggl\{C_4(C_S^{B,\Lambda})^{\frac{1}{\nu-1}}\tau \Biggl[ \frac{1+\tau^{-1}}{ (\sigma-\sigma')^2}\Biggr]^{\frac{\nu}{\nu-1}}\Biggr\}^{1/\alpha} \| v\|_{\alpha,V_{\sigma},\Lambda},
\]
for all $\alpha>0$ and $1/2\leq \sigma'<\sigma\leq 1$. Since by Theorem \ref{thm:loginequalities} we have
\[
\gamma^{\Lambda}\{(t,z)\in V\, |\,\log v>\ell\}\leq   C_6 \, \gamma^{\Lambda}(V) \tau^{-1} \Bigl[M^{B,\Lambda} (C_{P}^{B,\Lambda}\vee \tau^2)\Bigr]\ell^{-1},
\]
then Bombieri-Giusti's lemma is applicable and yields
\[
\sup_{Q_+} e^{-\kappa} u^{-1} \lesssim C_{BG}^B
\]
for some $C_{BG}^B$ which we can assume to be the same as before taking the maximum of the two.
 	
Putting the two inequalities together gives the result.
\end{proof}

\begin{theorem}[Parabolic Harnack inequality]  Fix $\tau>0$ and $\delta\in[1/2,1)$. Let $u$ be any positive caloric function in $Q= (s-\tau r^2,s)\times B(x,r)$. Then we have
	\begin{equation}\label{ineq:parabolicharnack}
	\sup_{Q_-} u(t,z)\leq C^{B,\Lambda}_H \inf_{Q_+} u(t,z)
	\end{equation}
	where the constant $C^{B,\Lambda}_H$ depends increasingly on $C_S^{B,\Lambda},C_P^{B,\Lambda},M^{B,\Lambda}$, and on $\tau,p,q,d,\delta$.
\end{theorem}
\begin{proof} It follows from the previous theorem for positive supercaloric functions and Corollary \ref{thm:mvi_cor}. 
\end{proof}

We have to remark that the constant appearing in \eqref{ineq:parabolicharnack} it is strongly dependent on the ball $B$ we are considering, in particular depends on its center and its radius. We use here the power of assumption \ref{ass:b.2} to get rid of this dependence for balls which are large enough as it was discussed in Section \ref{subsec:constants}.

Indeed for all $x\in\R^d$ we can find  $s(x,1)\geq 1$ such that  $C_H^{B,\Lambda}\leq 2 C^{*,\Lambda}_{H}$ for all $B(x,r)$ with $r>s(x,1)$.

\begin{theorem}[H\"older continuity]\label{thm:holdercontinuity} Let $x\in \R^d$, and $s(x,1)\geq 1$ as above. Let $r>s(x,1)$ and $\sqrt{t}\geq r$. Define $t_0 := t+1$ and $r_0 := \sqrt{t_0}$. If $u$ is a positive caloric function on $(0,t_0)\times B(x,r_0)$ then for all $z,y\in B(x,r)$ we have
\begin{equation}\label{ineq:holderineq}
u(t,z)-u(t,y) \leq c\biggl(\frac{r}{\sqrt{t}}\biggr)^{\theta} \sup_{[3t_0/4,t_0]\times B(x,\sqrt{t_0}/2)} u
\end{equation}
where $\theta, c$ are constants which depends only on $C^{*,\Lambda}_{H}$.
\end{theorem}
\begin{proof}
	Set $r_k\defeq 2^{-k} r_0$ and let
	\[
	Q_k\defeq (t_0-r_k^2,t_0)\times B(x,r_k),
	\]
	let $Q^-_k$ and $Q^+_k$ be accordingly defined as in \eqref{definitionsets} with $\delta=1/2$ and $\tau=1$,
	\[
	Q_k^-\defeq (t_0-7/8 r_k^2,t_0-5/8r_k^2)\times B(x,1/2r_k),\quad Q_k^+\defeq (t_0-1/4r_k^2,t_0)\times B(x,1/2r_k).	
	\] 
	Notice that $Q_{k+1}\subset Q_k$ and actually $Q_{k+1}=Q_k^+$. We set
	\[
	v_k=\frac{u-\inf_{Q_k} u}{\sup_{Q_k}u- \inf_{Q_k} u}
	\]
	clearly $v_k$ is a caloric on $Q_k$, in particular $0\leq v_k\leq 1$ and 
	\[
	\osc(v_k,Q_k):=\sup_{Q_k} v_k - \inf_{Q_k} v_k=1
	\]
	this implies that replacing $v_k$ by $1-v_k$ if necessary $\sup_{Q_k^-} v_k \geq 1/2$. Now for all $k$ such that $r_k\geq s(x,1)$ we can apply the parabolic Harnack inequality and get
	\[
	\frac{1}{2}\leq \sup_{Q_k^-} v_k \leq 2 C_{H}^{*,\Lambda} \inf_{Q_k^+} v_k
	\]
	since $Q_k^+ = Q_{k+1}$ we have that
	\begin{align*}
	\osc(u,Q_{k+1}) &= \frac{\sup_{Q_{k+1}}u-\inf_{Q_{k+1}}  u}{osc(u,Q_k)} \osc(u,Q_k)\\
	&=\biggl(\frac{\sup_{Q_{k+1}}u-\inf_{Q_{k}}u}{\osc(u,Q_k)}-\inf_{Q_{k+1}} v_k\biggr)\osc(u,Q_k)
	\end{align*}
	this yields $\osc(u,Q_{k+1})\leq (1-\delta) \osc(u,Q_k)$ with $\delta^{-1}=4C^{*,\Lambda}_H$. We can now iterate the inequality up to $k_0$ such that $r_{k_0}\geq r > r_{k_0+1}$ and get
	\[
	\osc(u,Q_{k_0})\leq (1-\delta)^{k_0-1} \osc(u,Q_0^+)
	\]
	Finally since $B(x,r)\subset B(x,r_{k_0})$ and $t\in (t_0-r_{k_0}^2,t_0)$ the claim is proved.
\end{proof}

Starting from \eqref{ineq:holderineq} and knowing that $p_t(z,\cdot)$ is caloric on the whole $\R^d$ for almost all $z\in\R^d$ we get the following corollary.
\begin{corollary}\label{lem:holder}
	 Let $x\in \R^d$, and $s(x,1)\geq 1$ as above. Let $r>s(x,1)\geq 1$ and $\sqrt{t}\geq r$. Then we have that for almost all $o\in \R^d$
		\begin{equation}\label{ineq:boundholder}
		\sup_{z,y\in B(x,r)}|p_t(o,z)-p_t(o,y)| \leq c\biggl( \frac{r}{\sqrt{t}}\biggr)^{\theta} t^{-d/2}
		\end{equation}
		where $\theta, c$ are positive constants which depends only on $C^{*,\Lambda}_{H}$.
\end{corollary}
\begin{proof} We have just to bound the right hand side of \eqref{ineq:holderineq}.  Define $t_0=t+1$ as in the previous theorem. By the Harnack's inequality applied to the caloric function $p_t(o,\cdot)$ we have
	\begin{align*}
 \sup_{[3t_0/4,t_0]\times B(x,\sqrt{t_0}/2)} p_s(o,u)&\leq 2 C_H^{*,\Lambda}\inf_{[3/2 t_0,7/4t_0]\times B(x,\sqrt{t_0}/2)} p_s(o,u)\\
 &\leq  2 C_H^{*,\Lambda}\Bigl[|B(x,\sqrt{t_0}/2)|\|\Lambda\|_{1,B(x,\sqrt{t_0}/2)}\Bigr]^{-1}\int_{B(x,\sqrt{t_0}/2)} p_{\bar{t}} (o,u)\Lambda(u)\,du
	\end{align*}
	where $\bar{t}\in [3/2 t_0,7/4t_0]$.
	
Clearly $\int_{B(x,\sqrt{t_0}/2)} p_{\bar{t}} (o,u)\Lambda(u)\,du\leq 1$. For $\sqrt{t_0}>r>s(x,1)$, we can bound  $\|\Lambda\|_{1,B(x,\sqrt{t_0}/2)}$ by a constant which does not depend on $x$ or $t_0$ by assumption \ref{ass:b.2}, hence we finally get the desired bound. 
\end{proof}

We want to stress that Corollary \eqref{lem:holder} is not a true H\"older continuity result, since we cannot bound the variations for arbitrarily small balls, and indeed it is not even possible to prove continuity of the density with this technique. We are interested in finding H\"older's continuity bounds for $p_{t/\epsilon^2}(o,\cdot/\epsilon)$ for almost all $x\in \R^d$, uniformly for $\epsilon$ small. In order to do that we need the following assumption, which accounts for a control of moving averages.

\begin{itemize}
\item[$(b.3)$\namedlabel{ass:b.3}{$(b.3)$}]there exist $p,q\in[1,\infty]$ satisfying  $1/p+1/q<2/d$ such that for all $r>0$ and $x\in \R^d$
 \[
  \sup_{x\in\R^d}\limsup_{\epsilon\to0} \frac{1}{|B(x/\epsilon,r/\epsilon)|} \int_{B(x/\epsilon,r/\epsilon)} \Lambda^p +\lambda^{-q} \,dx <\infty.
 \]
\end{itemize}
It is clear that \ref{ass:b.3} implies \ref{ass:b.2} choosing $x=0$.
Given assumption \ref{ass:b.3} it is not surprising that the following holds true.

\begin{lemma} Fix $r>0$, $\sqrt{t}\geq r$, $x\in\R^d$ and $s(x,1)\geq 1$ as above. Let $\epsilon>0$ such that $r/\epsilon > s(x,1)$, then we have that for almost all $o\in \R^d$
		\begin{equation}\label{ineq:boundholderepsilon}
		\sup_{z,y\in B(x,r)}
		\epsilon^{-d}|p_{t/\epsilon^2}(o,z/\epsilon)-p_{t/\epsilon^2}(o,y/\epsilon)| \leq c\biggl( \frac{r}{\sqrt{t}}\biggr)^{\theta} t^{-d/2}
		\end{equation}
		where $\theta, c$ are positive constants which depends only on $C^{*,\Lambda}_{H}$. 
\end{lemma}
\begin{proof} The proof is the same as in Theorem \ref{thm:holdercontinuity} since given a caloric function $u(t,x)$, then $u(t/\epsilon^2,x/\epsilon)$ is caloric with respect to the Dirichlet form with coefficients given by $a(x/\epsilon)$. Assumption \ref{ass:b.3} is used to have uniform constants for moving averages. 
\end{proof}

From the H\"older continuity estimates \eqref{ineq:boundholderepsilon} we get
\begin{lemma}\label{lemma:lclt} For almost all $o\in\R^d$ and all $r>0$
\[
\lim_{r_0\to 0} \limsup_{\epsilon\to0} \sup_{\substack{x,y\in B(o,r)\\|x-y|< r_0}}\sup_{t\in I} \epsilon^{-d}|p(t/\epsilon^2,o,x/\epsilon)-p(t/\epsilon^2,o,y/\epsilon)|=0
\]
\end{lemma}
\begin{proof}
Let us denote by $t_1:= \inf I$. Fix $\delta>0$ and set
\[
r_0 := \frac{\sqrt{t_1}}{2}\wedge\biggl( \frac{ t_1^{d/2}\delta}{2 c}\biggr)^{1/\theta}\sqrt{t_1}.
\]
Since $B(o,r)$ is compact we can cover it by a finite set of balls $\{B(x,r_0/2)\}_{x\in\mathcal{X}}$  of radius $r_0/2$ and centers $x\in\mathcal{X}$. We can now find $\bar{\epsilon}>0$ such that for all $\epsilon<\bar{\epsilon}$ we have $r_0/\epsilon \geq \max_{x\in \mathcal{X}} s(x,1)$. Now an application of \eqref{ineq:boundholderepsilon} gives
\[
\sup_{x\in\mathcal{X}}\sup_{|x-y|< r_0}\sup_{t\in I} \epsilon^{-d}|p(t/\epsilon^2,o,x/\epsilon)-p(t/\epsilon^2,o,y/\epsilon)|\leq c\biggl( \frac{r_0}{\sqrt{t_1}}\biggr)^{\theta} t_1^{-d/2}\leq \frac{\delta}{2}.
\]
Next we can use this bound to conclude, take $z\in B(o,r)$, and take $x\in \mathcal{X}$ such that $|z-x|< r_0/2$, then
\begin{align*}
\sup_{|z-y|\leq r_0/2}\sup_{t\in I} \epsilon^{-d}&|p(t/\epsilon^2,o,z/\epsilon)-p(t/\epsilon^2,o,y/\epsilon)|\\
&\leq\sup_{|z-x|\leq r_0}\sup_{t\in I} \epsilon^{-d}|p(t/\epsilon^2,o,x/\epsilon)-p(t/\epsilon^2,o,z/\epsilon)|\\ &+\sup_{|y-x|\leq r_0}\sup_{t\in I} \epsilon^{-d}|p(t/\epsilon^2,o,x/\epsilon)-p(t/\epsilon^2,o,y/\epsilon)|\leq \delta
\end{align*}
and this ends the proof since we showed that the bound is uniform in $z\in B(o,r)$.
\end{proof}

\section{Local Central Limit Theorem}
We finally give the main application of the computations we have developed in the preceding sections.
The approach we exploit is the one in \cite{CroydonHambly}, in particular their Assumption (4) must be compared with our inequality \eqref{ineq:holderineq}.

We denote by $k_t^{\Sigma}(x)$, $x\in\R^d$ the gaussian kernel with covariance matrix $\Sigma$, namely
\[
k_t^{\Sigma}(x) := \frac{1}{\sqrt{(2\pi t)^{d} \det \Sigma}} \exp\Bigl(-\frac{x\cdot \Sigma^{-1} x}{2t}\Bigr)
\]
We need here two further assumptions
\begin{itemize}
\item[$(b.4)$\namedlabel{ass:b.4}{$(b.4)$}] for almost all $x\in\R^d$ and all $r>0$
\[
\lim_{\epsilon\to 0} \frac{1}{|B(x/\epsilon,r/\epsilon)|} \int_{B(x/\epsilon,r/\epsilon)} \Lambda\,dx =: a_\Lambda < \infty.
\]
\item[$(b.5)$\namedlabel{ass:b.5}{$(b.5)$}]there exists a positive define symmetric matrix $\Sigma$ such that for almost all $o\in\R^d$, for any compact interval $I\subset(0,\infty)$, almost all $x\in\R^d$ and $r>0$
\[
\lim_{\epsilon\to 0} \frac{1}{\epsilon^d}\int_{B(x,r)}p_{t/\epsilon^2}(o,y/\epsilon)\,\Lambda(y/\epsilon)dy\to\int_{B(x,r)} k^{\Sigma}_t(y)\,dy
\]
uniformly in $t\in I$.
\end{itemize}

\begin{theorem}\label{thm:localCLT} Fix a compact interval $I\subset (0,\infty)$ and $r>0$. Assume \ref{ass:b.1}--\ref{ass:b.5}, then for almost all $o\in\R^d$ and for all $r>0$
\[
\lim_{\epsilon\to0} \sup_{x\in B(o,r)}\sup_{t\in I} |\epsilon^{-d} p(t/\epsilon^2,o,x/\epsilon)-a_\Lambda^{-1}k_t^\Sigma(x)|=0.
\]
\end{theorem}
\begin{proof}
The proof presented here is a slight variation of the one in  \cite{CroydonHambly} due to the fact that we work on $\R^d$ rather than on graphs. For $x\in B(o,r)$ and $r_0>0$ denote 
\[
J(t,\epsilon):=\frac{1}{\epsilon^d}\int_{B(x,r_0)}p_{t/\epsilon^2}(o,y/\epsilon)\,\Lambda(y/\epsilon)dy-\int_{B(x,r_0)} k_t(y)\,dy
\]
where $k_t := k_t^\Sigma$ from assumption \ref{ass:b.5} is the gaussian kernel with covariance matrix $\Sigma$. Then we can split $J(t,\epsilon) =J_1(t,\epsilon)+J_2(t,\epsilon)+J_3(t,\epsilon)+J_4(t,\epsilon)$ 
\begin{align*}
J_1(t,\epsilon) &:= \int_{\frac{1}{\epsilon}  B(x,r_0)}[p(t/\epsilon^2,o,y)-p^\omega(t/\epsilon^2,o,x/\epsilon)] \Lambda(y)dy\\
J_2(t,\epsilon)&:=\int_{\frac{1}{\epsilon}  B(x,r_0)} \Lambda(y)dy \, \Bigl[p(t/\epsilon^2,o,x/\epsilon)-\epsilon^d a_{\Lambda}^{-1}k_t(x)\Bigr]\\
J_3(t,\epsilon)&:=k_t(x)\Bigl[\epsilon^d a_{\Lambda}^{-1}\int_{\frac{1}{\epsilon}  B(x,r_0)}\Lambda(y)dy-|B(x,r_0)|\Bigr]\\
J_4(t,\epsilon)&:=\int_{B(x,r_0)}(k_t(x)-k_t(y))dy.
\end{align*}
Fix $\delta>0$. Thanks to Lemma \ref{lemma:lclt} and by the continuity of $k_t$ we can chose $r_0\in (0,1)$ and $\bar{\epsilon}>0$ small such that for all $\epsilon<\bar{\epsilon}$
\begin{equation}\label{ineq:holder}
 \sup_{\substack{x,y\in B(o,r+1)\\|x-y|\leq r_0}}\sup_{t\in I} \frac{1}{\epsilon^d}|p(t/\epsilon^2,o,y/\epsilon)-p(t/\epsilon^2,o,x/\epsilon)|\leq \delta,
\end{equation}
and
\begin{equation}\label{ineq:kernel}
 \sup_{\substack{x,y\in B(0,r+1)\\|x-y|\leq r_0}}\sup_{t\in I} |k_t(y)-k_t(x)|\leq \delta.
\end{equation}
We can now easily bound $\sup_{t\in I}|J_4(t,\epsilon)|\leq \delta |B(x,r_0)|$. Furthermore, by assumption \ref{ass:b.4} taking $\bar{\epsilon}$ smaller if needed we get $\sup_{t\in I}|J_3(t,\epsilon)|\leq \delta |B(x,r_0)|$ for all $\epsilon\leq \bar{\epsilon}$. Exploiting \eqref{ineq:holder} we have a control on $J_1$. Namely $\sup_{t\in I}|J_1(t,\epsilon)|\leq \delta |B(x,r_0)|$. Finally by assumption \ref{ass:b.5} we have also that 
$\sup_{t\in I}|J(t,\epsilon)|\leq \delta |B(x,r_0)|$.

These estimates can be then used to control $|J_2(t,\epsilon)|$ for $\epsilon\leq \bar{\epsilon}$ uniformly in $t\in I$. Namely one gets
\[
\sup_{t\in I}| \epsilon^{-d}p(t/\epsilon^2,o,x/\epsilon)- a_\Lambda^{-1}k_t(x)|\leq 4\delta \biggl(\frac{\epsilon^d }{|B(x,r_0)|}\int_{\frac{1}{\epsilon}  B(x,r_0)} \Lambda(y)dy\biggr)^{-1}
\]
and we can take $\bar{\epsilon}$ even smaller to have by means of assumption \ref{ass:b.4}
\[
\biggl(\frac{\epsilon^d }{|B(x,r_0)|}\int_{\frac{1}{\epsilon}  B(x,r_0)} \Lambda(y)dy\biggr)^{-1} \leq (\delta+a_\Lambda).
\]
This gives for almost all $x\in R^d$
\[
\lim_{\epsilon\to 0}\sup_{t\in I}| \epsilon^{-d}p(t/\epsilon^2,o,x/\epsilon)- a_\Lambda^{-1}k_t(x)|=0.
\]
Consider now $r>0$ and $\delta>0$, and let $r_0\in(0,1)$ be chosen as before. Since $B(o,r)$ is compact there exists a finite covering $\{B(z,r_0)\}_{z\in\mathcal{X}}$ of $B(o,r)$ with $\mathcal{X}\subset B(o,r)$. Since $\mathcal{X}$ is finite, there exists $\bar{\epsilon}>0$ such that for all $\epsilon\leq\bar{\epsilon}$
\[
\sup_{z\in\mathcal{X}}\sup_{t\in I}| \epsilon^{-d}p(t/\epsilon^2,o,z/\epsilon)- a_\Lambda^{-1}k_t(z)|\leq \delta
\]
Next if $x\in B(o,r)$ then $x\in B(z,r_0)$ for some $z\in \mathcal{X}$ and we can write
\begin{align*}
\sup_{t\in I}| \epsilon^{-d}p(t/\epsilon^2,o,x/\epsilon)- a_\Lambda^{-1}k_t(x)| & \leq \sup_{t\in I}\epsilon^{-d}| p(t/\epsilon^2,o,x/\epsilon)- p(t/\epsilon^2,o,z/\epsilon)| \\
&+\sup_{t\in I}| \epsilon^{-d}p(t/\epsilon^2,o,z/\epsilon)- a_\Lambda^{-1}k_t(z)|\\
&+a_\Lambda^{-1}\sup_{t\in I}| k_t(x)- k_t(z)|.
\end{align*}
Since $x,z\in B(o,r+1)$ and $|x-z|\leq r_0$, inequality \eqref{ineq:kernel} implies that the last addendum is bounded by $\delta$, the second term is also bounded uniformly by $\delta$ since $z\in \mathcal{X}$. We can finally bound the first term uniformly by $\delta$ by means of \eqref{ineq:holder}. This ends the proof.
\end{proof}

\subsection{Application to Diffusions in Random Environment}

In this section we finally apply Theorem \ref{thm:localCLT} to obtain Theorem \ref{thm:localCLTRandom}

\begin{proof}[Proof of Theorem \ref{thm:localCLTRandom}]It is enough to show that assumptions \ref{ass:b.1}-\ref{ass:b.5} are satisfied for $\mu$-almost all realizations of the environment, then Theorem \ref{thm:localCLT} gives the result. 
	
By construction \ref{ass:a.1} implies \ref{ass:b.1} for $\mu$-almost all $\omega\in\Omega$. Assumption \ref{ass:a.2} together with the ergodic theorem \cite[Theorem 11.18]{landim} gives easily \ref{ass:b.2}-\ref{ass:b.4} $\mu$-almost surely. Finally \ref{ass:b.5} for $\mu$-almost all $\omega\in\Omega$ follows directly from \ref{ass:a.3}.

The second part of the statement follows easily since, if we assume  that $\lambda^\omega(\cdot)^{-1},\Lambda^\omega(\cdot)\in L^\infty_{loc}(\R^d)$ for $\mu$-almost all $\omega\in\Omega$, Theorem \ref{thm:localCLT} holds for all $o\in\R^d$, $\mu$-almost surely. Indeed, the density $p^\omega_t(x,y)$ is a continuous function of $x$ and $y$ by classical results in PDE theory \cite{gilbarg2001elliptic}.
\end{proof}
%


\appendix

\section{Dirichlet Forms}

Let $X$ be a locally compact metric separable space, and $m$ a positive Radon measure on $X$ such that $\supp[X]=m$. Consider the Hilbert space $L^2(X,m)$ with scalar product $\langle \cdot,\cdot\rangle$. We call a \emph{symmetric} form, a non-negative definite bilinear form $\E$ defined on a dense subset $\mathcal{D}(\E)\subset L^2(X,m)$. Given a symmetric form $(\E,\mathcal{D}(\E))$ on $L^2(X,m)$,  the form $\E_\beta\defeq \E+\beta \langle \cdot,\cdot\rangle$ defines a new symmetric form on $L^2(X,m)$ for each $\beta>0$. Note that $\mathcal{D}(\E)$ is a pre-Hilbert space with inner product $\E_\beta$. If $\mathcal{D}(\E)$ is complete with respect to $\E_\beta$, then $\E$ is said to be \emph{closed}. 

A closed symmetric form $(\E,\mathcal{D}(\E))$ on $L^2(X,m)$ is called a \emph{Dirichlet form} if it is Markovian, namely if for any given $u\in \mathcal{D}(\E)$, then $v=(0\vee u)\wedge1$ belongs to $\mathcal{D}(\E)$ and $\E(v,v)\leq \E(u,u)$.

We say that the Dirichlet form $(\E,\mathcal{D}(\E))$ on $L^2(X,m)$ is \emph{regular} if there is a subset  $\mathcal{H}$ of $\mathcal{D}(\E)\cap C_0(X)$ dense in $\mathcal{D}(\E)$ with respect to $\E_1$ and dense in $C_0(X)$ with respect to the uniform norm. $\mathcal{H}$ is called a \emph{core} for $\mathcal{D}(\E)$.

We say that the Dirichlet form $(\E,\mathcal{D}(\E))$ is \emph{local} if for all $u,v\in\mathcal{D}(\E)$ with disjoint compact support $\E(u,v)=0$. $\E$ is said \emph{strongly local} if $u,v\in \mathcal{D}(\E)$ with compact support and $v$ constant on a neighborhood of $\supp u$ implies $\E(u,v)=0$. 

\begin{lemma}\label{lem:cutoff} Let $B\subset \R^d$ and consider a cutoff $\eta\in C_0^\infty(B)$. Then, $u\in\F_{loc}\cup \F^\Lambda_{loc}$ implies $\eta u\in \F_B$.
\end{lemma}
\begin{proof} Take $u\in\F_{loc}^\Lambda$, then there exists $\bar{u}\in\F^\Lambda$ such that $u=\bar{u}$ on $2B$. Let $\{f_n\}_\N\subset C_0^\infty(\R^d)$ be such that $f_n\to \bar{u}$ with respect to $\E+\langle\cdot,\cdot\rangle_\Lambda$. Clearly $\eta f_n\in\F_B^\Lambda$ and $\eta f_n\to \eta \bar{u}=\eta u$ in $L^2(B,\Lambda dx)$. Moreover
\[
\E(\eta f_n-\eta f_m)\leq2\E(f_n-f_m)+\|\nabla\eta\|_\infty^2\int_B|f_n-f_m|^2\Lambda dx.
\]
Hence $\eta f_n$ is Cauchy in $L^2(B,\Lambda dx)$ with respect to $\E+\langle\cdot,\cdot\rangle_\Lambda$, which implies that $\eta u\in \F^\Lambda_B = \F_B$. If $u\in\F_{loc}$ the proof is similar, and one has only to observe that $\{f_n\}$ is Cauchy in $W^{2q/(q+1)}(B)$, which by Sobolev's embedding theorem implies that $\{f_n\}$ is Cauchy in $L^2(B,\Lambda dx)$.
\end{proof}

\begin{lemma}\label{lemma:composition} Let $\psi:\R\to \R$ be a globally Lipschitz function with constant $L$ and with $\psi(0)=0$. If $u\in\F$ ($\F_{loc}$), then $\psi(u)\in \F$ ($\F_{loc}$).
\end{lemma}

\begin{proof} Let $B\subset\R^d$ be a ball, and $\tilde{u}\in \F$ such that $u=\tilde{u}$ on $B$. 
Then, it is easy to verify that the function $\psi(\tilde{u})/L$ is a normal contraction of $\tilde{u}$, since
\[
|\psi(\tilde{u}(x))-\psi(\tilde{u}(y))|\leq L|\tilde{u}(x)-\tilde{u}(y)|,\quad |\psi(\tilde{u}(x))|\leq L |\tilde{u}(x)|.
\]
Hence $\psi(\tilde{u})\in\F$; in particular $\psi(\tilde{u})=\psi(u)$ on $B$, which gives the thesis.
\end{proof}

\begin{lemma} \label{lemma:parabolic}Let $F:\R\to\R$ be a twice differentiable  function with bounded second derivative and positive first derivative. Assume that $F'(0)=0$. Then for any caloric  (subcaloric, supercaloric) function $u$ we have
\[
\frac{d}{dt} ( F(u_t),\phi )_{\Lambda} + \E(u_t,F'(u_t)\phi)=0,\quad(\leq,\,\geq)
\]
for all $\phi\in C_0^\infty(\R^d)$, $\phi>0$ and $t>0$.
\end{lemma}
\begin{proof} Observe that $F''$ bounded and $F'(0)=0$ implies that $F'(u_t)\in \F^\Lambda$ by Lemma \ref{lemma:composition}.
\begin{align*}
\frac{d}{dt} ( &F(u_t),\phi )_{\Lambda} = \lim_{h\downarrow 0}\frac{1}{h}( F(u_{t+h})-F(u_{t}),\phi )_{\Lambda}\\
	&= \lim_{h\downarrow 0} \frac{1}{h}( F'(u_{t})(u_{t+h}-u_t),\phi)_{\Lambda} + \frac{1}{h}( R(u_{t+h}-u_t),\phi )_{\Lambda}.
\end{align*}
Where $|R(x)|\leq \|F''\|_\infty |x|^2$.
The first summand converges to $\E(u_t,F'(u_t)\phi)$ since $u_t$ solves \eqref{eq:parabolic}. It remains to show that the second summand goes to zero.
\[
\frac{1}{h}|( R(u_{t+h}-u_t),\phi )_{\Lambda}|\leq h\|\phi\|_\infty \|F''\|_\infty\|(u_{t+h}-u_t)h^{-1}\|^2_{2,\Lambda}\to 0
\]
as $h\to 0$. For subcaloric and supercaloric functions the proof follows the same lines.
\end{proof}
\section{Bombieri-Giusti's Lemma}

In order to obtain an Harnack inequality for positive weak solutions to an elliptic or parabolic equation we will make use of the following lemma due to Bombieri and Giusti, whose proof can be found in \cite{saloff2002aspects}.

Consider a collection of measurable subsets $U_{\sigma}$, $0 < \sigma \leq 1$, of a fixed measure space $(\mathcal{X},\mathcal{M})$ endowed with a measure $\gamma$, such that $U_{\sigma'}\subset U_\sigma$ whenever $\sigma' < \sigma$. In
our application, $U_\sigma$ will be $B(x,\sigma r)$ for some fixed ball $B(x,r)\subset\R^d$.
\begin{lemma}[Bombieri-Giusti \cite{Bombieri1971/72}]\label{lemma:BombieriGiusti} Fix $\delta \in(0,1)$. Let $\kappa$ and $K_1, K_2$ be positive constants and $0 < \alpha_0 \leq
\infty$. Let $u$ be a positive measurable function on $U\defeq U_1$ which satisfies
\begin{equation}\label{bombieri:conditionOne}\biggl(\int_{U_{\sigma'}} |u|^{\alpha_0}\,d\gamma\biggr)^{\frac{1}{\alpha_0}} \leq \Bigl(K_1 (\sigma-\sigma')^{-\kappa} \gamma(U)^{-1}\Bigr)^{\frac{1}{\alpha}-\frac{1}{\alpha_0}}\biggl(\int_{U_{\sigma}} |u|^{\alpha}\,d\gamma\biggr)^{\frac{1}{\alpha}} \
\end{equation}
for all $\sigma, \sigma'$ and $\alpha$ such that $0<\delta\leq \sigma'<\sigma\leq 1$ and $0<\alpha\leq \min\{1,\alpha_0/2\}$. Assume further that $u$ satisfies
\begin{equation}
\gamma(\log u>\ell)\leq K_2\,\gamma(U) \ell^{-1}
\end{equation}
for all $\ell>0$. Then
\[
\biggl(\int_{U_{\delta}} |u|^{\alpha_0}\,d\gamma\biggr)^{\frac{1}{\alpha_0}}\leq C_{BG}\, \gamma(U)^{\frac{1}{\alpha_0}},
\]
where $C_{BG}$ depends only on $K_1, K_2,\delta,\kappa$ and a lower bound on $\alpha_0$.
\end{lemma}

\bibliographystyle{plain}
\bibliography{bibliography}

\end{document}